\documentclass{amsart}

\usepackage{lmodern}
\usepackage[english]{babel}
\usepackage[latin1]{inputenc}
\usepackage{microtype}
\usepackage{amsfonts,amssymb,amsmath,amsthm,mathrsfs}
\usepackage{color, graphicx}
\usepackage[all]{xy}
\usepackage[bookmarks=false,backref=page]{hyperref}
\usepackage{version}
\usepackage{upgreek}
\usepackage{tikz}

 \newcount\grcalca
 \newcount\grcalcb
 \newcount\grcalcc
 \newcount\grcalcd
 \newcount\grcalce
 \newcount\grcalcf
 \newcount\grcalcg
 \newcount\grcalch
 \newcount\grcalci
 \newcount\grrow
 \newcount\grcolumn
 \newcount\grwidth
 \newcount\factor
 \newcount\hfactor
 \newcount\qfactor
 \newcount\tfactor
 \newcount\sfactor
 \newcount\dfactor
 \hfactor = \factor
 \divide    \hfactor by 2
 \qfactor = \hfactor
 \divide    \qfactor by 2
 \tfactor = \factor
 \divide    \tfactor by 3
 \sfactor = \factor
 \divide    \sfactor by 6
 \dfactor = \factor
 \divide    \dfactor by 12

 \newcommand{\gbeg}[2]{ 
   \unitlength=1pt
   \grrow = #2
   \grcolumn = 0
   \grcalca = #1
   \grcalcb = #2
   \multiply \grcalca by \factor
   \grwidth = \grcalca
   \multiply \grcalcb by \factor
   \begin{minipage}{\grcalca pt}
   \begin{picture}(\grcalca,\grcalcb)
   \advance \grcalcb by -\factor
	}
	
  \newcommand{\gend}{ 
   \end{picture}
   {\vskip2.5ex}
   \end{minipage} }

 \newcommand{\gnl}{ 
   \advance \grrow by -1
   \grcolumn = 0}

 \newcommand{\gvac}[1]{  
   \advance \grcolumn by #1} 
 
 \newcommand{\gcl}[1]{ 
   \grcalca = \grcolumn
   \multiply \grcalca by \factor
   \advance \grcalca by \hfactor
   \grcalcb = \grrow
   \multiply \grcalcb by \factor
   \grcalcc = #1
   \multiply \grcalcc by \factor
   \put(\grcalca,\grcalcb) {\line(0,-1){\grcalcc}} 
   \advance \grcolumn by 1}

 \newcommand{\gcn}[4]{ 
   \grcalca = \grcolumn
   \multiply \grcalca by \factor
   \grcalci = #3
   \multiply \grcalci by \hfactor
   \advance \grcalca by \grcalci
   \grcalcb = \grcolumn
   \multiply \grcalcb by \factor 
   \grcalci = #3
   \advance \grcalci by #4
   \multiply \grcalci by \qfactor
   \advance \grcalcb by \grcalci
   \grcalcc = \grcolumn
   \multiply \grcalcc by \factor
   \grcalci = #4
   \multiply \grcalci by \hfactor
   \advance \grcalcc by \grcalci
   \grcalcd = \grrow
   \multiply \grcalcd by \factor 
   \grcalce = \grrow
   \multiply \grcalce by \factor 
   \grcalci = #2
   \multiply \grcalci by \tfactor
   \advance \grcalce by -\grcalci
   \grcalcf = \grrow
   \multiply \grcalcf by \factor 
   \grcalci = #2
   \multiply \grcalci by \hfactor
   \advance \grcalcf by -\grcalci
   \grcalcg = \grrow
   \multiply \grcalcg by \factor 
   \grcalci = #2
   \multiply \grcalci by \tfactor
   \multiply \grcalci by 2
   \advance \grcalcg by -\grcalci
   \grcalch = \grrow
   \advance \grcalch by -#2
   \multiply \grcalch by \factor 
   \qbezier(\grcalca,\grcalcd)(\grcalca,\grcalce)(\grcalcb,\grcalcf) 
   \qbezier(\grcalcb,\grcalcf)(\grcalcc,\grcalcg)(\grcalcc,\grcalch) 
   \advance \grcolumn by #1}

 \newcommand{\gnot}[1]{ 
   \grcalca = \grcolumn
   \multiply \grcalca by \factor
   \advance \grcalca by \hfactor
   \grcalcb = \grrow
   \multiply \grcalcb by \factor
   \advance \grcalcb by -\hfactor
   \put(\grcalca,\grcalcb) {\makebox(0,0){$\scriptstyle #1$}} }

 \newcommand{\got}[2]{ 
   \grcalca = \grcolumn
   \multiply \grcalca by \factor
   \grcalcc = #1
   \multiply \grcalcc by \hfactor
   \advance \grcalca by \grcalcc
   \grcalcb = \grrow
   \multiply \grcalcb by \factor
   \advance \grcalcb by -\tfactor
   \advance \grcalcb by -\tfactor
   \put(\grcalca,\grcalcb){\makebox(0,0)[b]{$#2$}}
   \advance \grcolumn by #1}

 \newcommand{\gob}[2]{
   \grcalca = \grcolumn
   \multiply \grcalca by \factor
   \grcalcc = #1
   \multiply \grcalcc by \hfactor
   \advance \grcalca by \grcalcc
   \put(\grcalca,0){\makebox(0,0)[b]{$#2$}}
   \advance \grcolumn by #1}

 \newcommand{\gmu}{  
   \grcalca = \grcolumn
   \advance \grcalca by 1
   \multiply \grcalca by \factor
   \grcalcb = \grrow
   \multiply \grcalcb by \factor
   \grcalcc = \factor
   \advance \grcalcc by \hfactor
   \put(\grcalca,\grcalcb){\oval(\factor,\grcalcc)[b]}
   \advance \grcalcb by -\hfactor
   \advance \grcalcb by -\qfactor
   \put(\grcalca,\grcalcb) {\line(0,-1){\qfactor}} 
   \advance \grcolumn by 2}

 \newcommand{\gcmu}{   
   \grcalca = \grcolumn
   \advance \grcalca by 1
   \multiply \grcalca by \factor
   \grcalcb = \grrow
   \advance \grcalcb by -1
   \multiply \grcalcb by \factor
   \grcalcc = \factor
   \advance \grcalcc by \hfactor
   \put(\grcalca,\grcalcb){\oval(\factor,\grcalcc)[t]}
   \advance \grcalcb by \factor
   \put(\grcalca,\grcalcb) {\line(0,-1){\qfactor}} 
   \advance \grcolumn by 2}

 \newcommand{\glm}{  
   \grcalca = \grcolumn
   \multiply \grcalca by \factor
   \advance \grcalca by \hfactor
   \grcalcb = \grcalca
   \advance \grcalcb by \factor
   \grcalcc = \grrow
   \multiply \grcalcc by \factor
   \grcalcd = \grcalcc
   \advance \grcalcd by -\tfactor
   \grcalce = \grcalcd
   \advance \grcalce by -\tfactor
   \put(\grcalca, \grcalcc){\line(0,-1){\tfactor}}
   \put(\grcalca, \grcalcd){\line(1,0){\factor}}
   \put(\grcalca, \grcalcd){\line(3,-1){\factor}}
   \put(\grcalcb, \grcalcc){\line(0,-1){\factor}}
   \advance \grcolumn by 2}

 \newcommand{\grm}{ 
   \grcalcb = \grcolumn
   \multiply \grcalcb by \factor
   \advance \grcalcb by \hfactor
   \grcalca = \grcalcb
   \advance \grcalca by \factor
   \grcalcc = \grrow
   \multiply \grcalcc by \factor
   \grcalcd = \grcalcc
   \advance \grcalcd by -\tfactor
   \grcalce = \grcalcd
   \advance \grcalce by -\tfactor
   \put(\grcalca, \grcalcc){\line(0,-1){\tfactor}}
   \put(\grcalca, \grcalcd){\line(-1,0){\factor}}
   \put(\grcalca, \grcalcd){\line(-3,-1){\factor}}
   \put(\grcalcb, \grcalcc){\line(0,-1){\factor}}
   \advance \grcolumn by 2}

 \newcommand{\glcm}{ 
   \grcalca = \grcolumn
   \multiply \grcalca by \factor
   \advance \grcalca by \hfactor
   \grcalcb = \grcalca
   \advance \grcalcb by \factor
   \grcalcc = \grrow
   \advance \grcalcc by -1
   \multiply \grcalcc by \factor
   \grcalcd = \grcalcc
   \advance \grcalcd by \tfactor
   \grcalce = \grcalcd
   \advance \grcalce by \tfactor
   \put(\grcalca, \grcalcc){\line(0,1){\tfactor}}
   \put(\grcalca, \grcalcd){\line(1,0){\factor}}
   \put(\grcalca, \grcalcd){\line(3,1){\factor}}
   \put(\grcalcb, \grcalcc){\line(0,1){\factor}}
   \advance \grcolumn by 2}

 \newcommand{\grcm}{ 
   \grcalcb = \grcolumn
   \multiply \grcalcb by \factor
   \advance \grcalcb by \hfactor
   \grcalca = \grcalcb
   \advance \grcalca by \factor
   \grcalcc = \grrow
   \advance \grcalcc by -1
   \multiply \grcalcc by \factor
   \grcalcd = \grcalcc
   \advance \grcalcd by \tfactor
   \grcalce = \grcalcd
   \advance \grcalce by \tfactor
   \put(\grcalca, \grcalcc){\line(0,1){\tfactor}}
   \put(\grcalca, \grcalcd){\line(-1,0){\factor}}
   \put(\grcalca, \grcalcd){\line(-3,1){\factor}}
   \put(\grcalcb, \grcalcc){\line(0,1){\factor}}
   \advance \grcolumn by 2}

 \newcommand{\gwmu}[1]{    
   \grcalca = \grcolumn
   \multiply \grcalca by \factor
   \grcalcd = \hfactor
   \multiply \grcalcd by #1
   \advance \grcalca by \grcalcd
   \grcalcb = \grrow
   \multiply \grcalcb by \factor
   \grcalcc = \factor
   \advance \grcalcc by \hfactor
   \grcalcd = #1
   \advance \grcalcd by -1
   \multiply \grcalcd by \factor
   \put(\grcalca,\grcalcb){\oval(\grcalcd,\grcalcc)[b]}
   \advance \grcalcb by -\hfactor
   \advance \grcalcb by -\qfactor
   \put(\grcalca,\grcalcb) {\line(0,-1){\qfactor}} 
   \advance \grcolumn by #1}

 \newcommand{\gwcm}[1]{   
   \grcalca = \grcolumn
   \multiply \grcalca by \factor
   \grcalcd = \hfactor
   \multiply \grcalcd by #1
   \advance \grcalca by \grcalcd
   \grcalcb = \grrow
   \advance \grcalcb by -1
   \multiply \grcalcb by \factor
   \grcalcc = \factor
   \advance \grcalcc by \hfactor
   \grcalcd = #1
   \advance \grcalcd by -1
   \multiply \grcalcd by \factor
   \put(\grcalca,\grcalcb){\oval(\grcalcd,\grcalcc)[t]}
   \advance \grcalcb by \factor
   \put(\grcalca,\grcalcb) {\line(0,-1){\qfactor}} 
   \advance \grcolumn by #1}

 \newcommand{\gwmuc}[1]{    
   \grcalca = \grcolumn
   \multiply \grcalca by \factor
   \advance \grcalca by \hfactor
   \grcalcb = \grrow
   \multiply \grcalcb by \factor
   \grcalcc = #1
   \advance \grcalcc by -1
   \multiply \grcalcc by \factor
   \put(\grcalca,\grcalcb){\line(1,0){\grcalcc}}
   \advance \grcalca by -\hfactor
   \grcalcd = \hfactor
   \multiply \grcalcd by #1
   \advance \grcalca by \grcalcd
   \grcalcc = \factor
   \advance \grcalcc by \hfactor
   \grcalcd = #1
   \advance \grcalcd by -1
   \multiply \grcalcd by \factor
   \put(\grcalca,\grcalcb){\oval(\grcalcd,\grcalcc)[b]}
   \advance \grcalcb by -\hfactor
   \advance \grcalcb by -\qfactor
   \put(\grcalca,\grcalcb) {\line(0,-1){\qfactor}} 
   \advance \grcolumn by #1}

 \newcommand{\gwcmc}[1]{   
   \grcalca = \grcolumn
   \multiply \grcalca by \factor
   \advance \grcalca by \hfactor
   \grcalcb = \grrow
   \multiply \grcalcb by \factor
   \advance \grcalcb by -\factor
   \grcalcc = #1
   \advance \grcalcc by -1
   \multiply \grcalcc by \factor
   \put(\grcalca,\grcalcb){\line(1,0){\grcalcc}}
   \grcalcd = #1
   \advance \grcalcd by -1
   \multiply \grcalcd by \hfactor
   \advance \grcalca by \grcalcd
   \grcalcc = \factor
   \advance \grcalcc by \hfactor
   \grcalcd = #1
   \advance \grcalcd by -1
   \multiply \grcalcd by \factor
   \put(\grcalca,\grcalcb){\oval(\grcalcd,\grcalcc)[t]}
   \advance \grcalcb by \factor
   \put(\grcalca,\grcalcb) {\line(0,-1){\qfactor}} 
   \advance \grcolumn by #1}

 \newcommand{\gev}{  
   \grcalca = \grcolumn
   \advance \grcalca by 1
   \multiply \grcalca by \factor
   \grcalcb = \grrow
   \multiply \grcalcb by \factor
   \grcalcc = \factor
   \advance \grcalcc by \hfactor
   \put(\grcalca,\grcalcb){\oval(\factor,\grcalcc)[b]}
   \advance \grcolumn by 2}

 \newcommand{\gdb}{   
   \grcalca = \grcolumn
   \advance \grcalca by 1
   \multiply \grcalca by \factor
   \grcalcb = \grrow
   \advance \grcalcb by -1
   \multiply \grcalcb by \factor
   \grcalcc = \factor
   \advance \grcalcc by \hfactor
   \put(\grcalca,\grcalcb){\oval(\factor,\grcalcc)[t]}
   \advance \grcolumn by 2}

 \newcommand{\gwev}[1]{    
   \grcalca = \grcolumn
   \multiply \grcalca by \factor
   \grcalcd = \hfactor
   \multiply \grcalcd by #1
   \advance \grcalca by \grcalcd
   \grcalcb = \grrow
   \multiply \grcalcb by \factor
   \grcalcc = \factor
   \advance \grcalcc by \hfactor
   \grcalcd = #1
   \advance \grcalcd by -1
   \multiply \grcalcd by \factor
   \put(\grcalca,\grcalcb){\oval(\grcalcd,\grcalcc)[b]}
   \advance \grcolumn by #1}

 \newcommand{\gwdb}[1]{   
   \grcalca = \grcolumn
   \multiply \grcalca by \factor
   \grcalcd = \hfactor
   \multiply \grcalcd by #1
   \advance \grcalca by \grcalcd
   \grcalcb = \grrow
   \advance \grcalcb by -1
   \multiply \grcalcb by \factor
   \grcalcc = \factor
   \advance \grcalcc by \hfactor
   \grcalcd = #1
   \advance \grcalcd by -1
   \multiply \grcalcd by \factor
   \put(\grcalca,\grcalcb){\oval(\grcalcd,\grcalcc)[t]}
   \advance \grcolumn by #1}
	
 \newcommand{\gdual}{ 
	 \grcalca = \grcolumn
   \advance \grcalca by 1
   \multiply \grcalca by \factor
	 \advance \grcalca by \hfactor
   \grcalcb = \grrow
   \multiply \grcalcb by \factor
	 \advance \grcalcb by -\hfactor
   \grcalcc = \factor
	 \put(\grcalca,\grcalcb) {\line(2,-1){\grcalcc}}
	 \put(\grcalca,\grcalcb) {\line(-2,-1){\grcalcc}}
   \advance \grcolumn by 3 }
	
	\newcommand{\geval}{ 
	 \grcalca = \grcolumn
   \advance \grcalca by 1
   \multiply \grcalca by \factor
	 \advance \grcalca by \hfactor
   \grcalcb = \grrow
   \multiply \grcalcb by \factor
	 \advance \grcalcb by -\hfactor
   \grcalcc = \factor
	 \put(\grcalca,\grcalcb) {\line(2,1){\grcalcc}}
	 \put(\grcalca,\grcalcb) {\line(-2,1){\grcalcc}}
   \advance \grcolumn by 3 }

 \newcommand{\gbr}{ 
   \grcalca = \grcolumn
   \multiply \grcalca by \factor
   \advance \grcalca by \hfactor
   \grcalcb = \grcalca
   \advance \grcalcb by \hfactor
   \grcalcc = \grcalca
   \advance \grcalcc by \factor
   \grcalcd = \grrow
   \multiply \grcalcd by \factor
   \grcalce = \grcalcd
   \advance \grcalce by -\tfactor
   \grcalcf = \grcalcd
   \advance \grcalcf by -\hfactor
   \grcalcg = \grcalce
   \advance \grcalcg by -\tfactor
   \grcalch = \grcalcd
   \advance \grcalch by -\factor
   \qbezier(\grcalca,\grcalcd)(\grcalca,\grcalce)(\grcalcb,\grcalcf) 
   \qbezier(\grcalcb,\grcalcf)(\grcalcc,\grcalcg)(\grcalcc,\grcalch) 
   \advance \grcalcf by -\dfactor
   \advance \grcalcb by -\sfactor
   \qbezier(\grcalca,\grcalch)(\grcalca,\grcalcg)(\grcalcb,\grcalcf) 
   \advance \grcalcf by \sfactor
   \advance \grcalcb by \tfactor
   \qbezier(\grcalcc,\grcalcd)(\grcalcc,\grcalce)(\grcalcb,\grcalcf) 
   \advance \grcolumn by 2}

 \newcommand{\gibr}{  
   \grcalca = \grcolumn
   \multiply \grcalca by \factor
   \advance \grcalca by \hfactor
   \grcalcb = \grcalca
   \advance \grcalcb by \hfactor
   \grcalcc = \grcalca
   \advance \grcalcc by \factor
   \grcalcd = \grrow
   \multiply \grcalcd by \factor
   \grcalce = \grcalcd
   \advance \grcalce by -\tfactor
   \grcalcf = \grcalcd
   \advance \grcalcf by -\hfactor
   \grcalcg = \grcalce
   \advance \grcalcg by -\tfactor
   \grcalch = \grcalcd
   \advance \grcalch by -\factor
   \qbezier(\grcalcc,\grcalcd)(\grcalcc,\grcalce)(\grcalcb,\grcalcf) 
   \qbezier(\grcalcb,\grcalcf)(\grcalca,\grcalcg)(\grcalca,\grcalch) 
   \advance \grcalcf by -\dfactor
   \advance \grcalcb by \sfactor
   \qbezier(\grcalcc,\grcalch)(\grcalcc,\grcalcg)(\grcalcb,\grcalcf) 
   \advance \grcalcf by \sfactor
   \advance \grcalcb by -\tfactor
   \qbezier(\grcalca,\grcalcd)(\grcalca,\grcalce)(\grcalcb,\grcalcf) 
   \advance \grcolumn by 2}

\newcommand{\gsy}{  
   \grcalca = \grcolumn
   \multiply \grcalca by \factor
   \advance \grcalca by \hfactor
   \grcalcb = \grcalca
   \advance \grcalcb by \hfactor
   \grcalcc = \grcalca
   \advance \grcalcc by \factor
   \grcalcd = \grrow
   \multiply \grcalcd by \factor
   \grcalce = \grcalcd
   \advance \grcalce by -\tfactor
   \grcalcf = \grcalcd
   \advance \grcalcf by -\hfactor
   \grcalcg = \grcalce
   \advance \grcalcg by -\tfactor
   \grcalch = \grcalcd
   \advance \grcalch by -\factor
   \qbezier(\grcalcc,\grcalcd)(\grcalcc,\grcalce)(\grcalcb,\grcalcf) 
   \qbezier(\grcalcb,\grcalcf)(\grcalca,\grcalcg)(\grcalca,\grcalch) 
   \advance \grcalcf by -\dfactor
   \advance \grcalcb by \sfactor
   \qbezier(\grcalcc,\grcalch)(\grcalcc,\grcalcg)(\grcalcb,\grcalcf) 
   \qbezier(\grcalca,\grcalcd)(\grcalca,\grcalce)(\grcalcb,\grcalcf) 
   \advance \grcolumn by 2}

 \newcommand{\gbrc}{ 
   \grcalca = \grcolumn
   \multiply \grcalca by \factor
   \advance \grcalca by \hfactor
   \grcalcb = \grcalca
   \advance \grcalcb by \hfactor
   \grcalcc = \grcalca
   \advance \grcalcc by \factor
   \grcalcd = \grrow
   \multiply \grcalcd by \factor
   \grcalce = \grcalcd
   \advance \grcalce by -\tfactor
   \grcalcf = \grcalcd
   \advance \grcalcf by -\hfactor
   \grcalcg = \grcalce
   \advance \grcalcg by -\tfactor
   \grcalch = \grcalcd
   \advance \grcalch by -\factor
   \put(\grcalcb,\grcalcf){\circle{\hfactor}}
   \qbezier(\grcalca,\grcalcd)(\grcalca,\grcalce)(\grcalcb,\grcalcf) 
   \qbezier(\grcalcb,\grcalcf)(\grcalcc,\grcalcg)(\grcalcc,\grcalch) 
   \advance \grcalcf by -\dfactor
   \advance \grcalcb by -\sfactor
   \qbezier(\grcalca,\grcalch)(\grcalca,\grcalcg)(\grcalcb,\grcalcf) 
   \advance \grcalcf by \sfactor
   \advance \grcalcb by \tfactor
   \qbezier(\grcalcc,\grcalcd)(\grcalcc,\grcalce)(\grcalcb,\grcalcf) 
   \advance \grcolumn by 2}

 \newcommand{\gibrc}{ 
   \grcalca = \grcolumn
   \multiply \grcalca by \factor
   \advance \grcalca by \hfactor
   \grcalcb = \grcalca
   \advance \grcalcb by \hfactor
   \grcalcc = \grcalca
   \advance \grcalcc by \factor
   \grcalcd = \grrow
   \multiply \grcalcd by \factor
   \grcalce = \grcalcd
   \advance \grcalce by -\tfactor
   \grcalcf = \grcalcd
   \advance \grcalcf by -\hfactor
   \grcalcg = \grcalce
   \advance \grcalcg by -\tfactor
   \grcalch = \grcalcd
   \advance \grcalch by -\factor
   \put(\grcalcb,\grcalcf){\circle{\hfactor}}
   \qbezier(\grcalcc,\grcalcd)(\grcalcc,\grcalce)(\grcalcb,\grcalcf) 
   \qbezier(\grcalcb,\grcalcf)(\grcalca,\grcalcg)(\grcalca,\grcalch) 
   \advance \grcalcf by -\dfactor
   \advance \grcalcb by \sfactor
   \qbezier(\grcalcc,\grcalch)(\grcalcc,\grcalcg)(\grcalcb,\grcalcf) 
   \advance \grcalcf by \sfactor
   \advance \grcalcb by -\tfactor
   \qbezier(\grcalca,\grcalcd)(\grcalca,\grcalce)(\grcalcb,\grcalcf) 
   \advance \grcolumn by 2}

 \newcommand{\gu}[1]{ 
   \grcalca = \grcolumn
   \multiply \grcalca by \factor
   \grcalcd = \hfactor
   \multiply \grcalcd by #1
   \advance \grcalca by \grcalcd
   \grcalcb = \grrow
   \advance \grcalcb by -1
   \multiply \grcalcb by \factor
   \put(\grcalca,\grcalcb) {\line(0,1){\hfactor}} 
   \advance \grcalcb by \hfactor
   \put(\grcalca,\grcalcb) {\circle*{3}}
   \advance \grcolumn by #1}

 \newcommand{\gcu}[1]{ 
   \grcalca = \grcolumn
   \multiply \grcalca by \factor
   \grcalcd = \hfactor
   \multiply \grcalcd by #1
   \advance \grcalca by \grcalcd
   \grcalcb = \grrow
   \multiply \grcalcb by \factor
   \put(\grcalca,\grcalcb) {\line(0,-1){\hfactor}} 
   \advance \grcalcb by -\hfactor
   \put(\grcalca,\grcalcb) {\circle*{3}}
   \advance \grcolumn by #1}

 \newcommand{\gmp}[1]{ 
   \grcalca = \grcolumn
   \multiply \grcalca by \factor
   \advance \grcalca by \hfactor
   \grcalcb = \grrow
   \multiply \grcalcb by \factor
   \put(\grcalca,\grcalcb) {\line(0,-1){\dfactor}} 
   \advance \grcalcb by -\factor
   \put(\grcalca,\grcalcb) {\line(0,1){\dfactor}} 
   \advance \grcalcb by \hfactor
   \grcalcc = \factor
   \advance \grcalcc by -\qfactor
   \put(\grcalca,\grcalcb) {\circle{\grcalcc}}
   \put(\grcalca,\grcalcb) {\makebox(0,0){$\scriptstyle #1$}}
   \advance \grcolumn by 1}
   
\newcommand{\gmpcu}[1]{ 
   \grcalca = \grcolumn
   \multiply \grcalca by \factor
   \advance \grcalca by \hfactor
   \grcalcb = \grrow
   \multiply \grcalcb by \factor
   \put(\grcalca,\grcalcb) {\line(0,-1){\dfactor}} 
   \advance \grcalcb by -\factor
   \advance \grcalcb by \hfactor
   \grcalcc = \factor
   \advance \grcalcc by -\qfactor
   \put(\grcalca,\grcalcb) {\circle{\grcalcc}}
   \put(\grcalca,\grcalcb) {\makebox(0,0){$\scriptstyle #1$}}
   \advance \grcolumn by 1}
   
\newcommand{\gmpu}[1]{ 
   \grcalca = \grcolumn
   \multiply \grcalca by \factor
   \advance \grcalca by \hfactor
   \grcalcb = \grrow
   \multiply \grcalcb by \factor
   \advance \grcalcb by -\factor
   \put(\grcalca,\grcalcb) {\line(0,1){\dfactor}} 
   \advance \grcalcb by \hfactor
   \grcalcc = \factor
   \advance \grcalcc by -\qfactor
   \put(\grcalca,\grcalcb) {\circle{\grcalcc}}
   \put(\grcalca,\grcalcb) {\makebox(0,0){$\scriptstyle #1$}}
   \advance \grcolumn by 1}  
			
\newcommand{\gmpcub}[1]{ 
   \grcalca = \grcolumn
   \multiply \grcalca by \factor
   \advance \grcalca by \hfactor
   \grcalcb = \grrow
   \multiply \grcalcb by \factor
   \put(\grcalca,\grcalcb) {\line(0,-1){\dfactor}} 
   \advance \grcalcb by -\factor
   \advance \grcalcb by \hfactor
   \grcalcc = \Dfactor
   \advance \grcalcc by -\hfactor
	 \advance \grcalcc by -\qfactor
	 \advance \grcalcb by -\qfactor
	 \advance \grcalcb by \dfactor
   \put(\grcalca,\grcalcb) {\circle{\grcalcc}}
   \put(\grcalca,\grcalcb) {\makebox(0,0){$\scriptstyle #1$}}
   \advance \grcolumn by 1}
	
	\newcommand{\gmpub}[1]{ 
   \grcalca = \grcolumn
   \multiply \grcalca by \factor
   \advance \grcalca by \hfactor
   \grcalcb = \grrow
   \multiply \grcalcb by \factor
   \advance \grcalcb by -\factor
   \put(\grcalca,\grcalcb) {\line(0,1){\dfactor}} 
   \advance \grcalcb by \hfactor
   \grcalcc = \factor
   \advance \grcalcc by -\sfactor
   \put(\grcalca,\grcalcb) {\circle{\grcalcc}}
   \put(\grcalca,\grcalcb) {\makebox(0,0){$\scriptstyle #1$}}
   \advance \grcolumn by 1}  
 \newcommand{\gbmp}[1]{ 
   \grcalca = \grcolumn
   \multiply \grcalca by \factor
   \advance \grcalca by \hfactor
   \grcalcb = \grrow
   \multiply \grcalcb by \factor
   \put(\grcalca,\grcalcb) {\line(0,-1){\dfactor}} 
   \advance \grcalcb by -\factor
   \put(\grcalca,\grcalcb) {\line(0,1){\dfactor}} 
   \advance \grcalca by -\hfactor
   \advance \grcalca by \dfactor
   \advance \grcalcb by \dfactor
   \grcalcc = \factor
   \advance \grcalcc by -\sfactor
   \put(\grcalca,\grcalcb) {\framebox(\grcalcc,\grcalcc){$\scriptstyle #1$}}
   \advance \grcolumn by 1}

 \newcommand{\gbmpt}[1]{
   \grcalca = \grcolumn
   \multiply \grcalca by \factor
   \advance \grcalca by \hfactor
   \grcalcb = \grrow
   \multiply \grcalcb by \factor
   \put(\grcalca,\grcalcb) {\line(0,-1){\dfactor}} 
   \advance \grcalcb by -\factor
   \advance \grcalca by -\hfactor
   \advance \grcalca by \dfactor
   \advance \grcalcb by \dfactor
   \grcalcc = \factor
   \advance \grcalcc by -\sfactor
   \put(\grcalca,\grcalcb) {\framebox(\grcalcc,\grcalcc){$\scriptstyle #1$}}
   \advance \grcolumn by 1}

 \newcommand{\gbmpb}[1]{
   \grcalca = \grcolumn
   \multiply \grcalca by \factor
   \advance \grcalca by \hfactor
   \grcalcb = \grrow
   \multiply \grcalcb by \factor
   \advance \grcalcb by -\factor
   \put(\grcalca,\grcalcb) {\line(0,1){\dfactor}} 
   \advance \grcalca by -\hfactor
   \advance \grcalca by \dfactor
   \advance \grcalcb by \dfactor
   \grcalcc = \factor
   \advance \grcalcc by -\sfactor
   \put(\grcalca,\grcalcb) {\framebox(\grcalcc,\grcalcc){$\scriptstyle #1$}}
   \advance \grcolumn by 1}

 \newcommand{\gbmpn}[1]{
   \grcalca = \grcolumn
   \multiply \grcalca by \factor
   \advance \grcalca by \hfactor
   \grcalcb = \grrow
   \multiply \grcalcb by \factor
   \advance \grcalcb by -\factor
   \advance \grcalca by -\hfactor
   \advance \grcalca by \dfactor
   \advance \grcalcb by \dfactor
   \grcalcc = \factor
   \advance \grcalcc by -\sfactor
   \put(\grcalca,\grcalcb) {\framebox(\grcalcc,\grcalcc){$\scriptstyle #1$}}
   \advance \grcolumn by 1}

 \newcommand{\glmptb}{ 
   \grcalca = \grcolumn
   \multiply \grcalca by \factor
   \advance \grcalca by \hfactor
   \grcalcb = \grrow
   \multiply \grcalcb by \factor
   \put(\grcalca,\grcalcb) {\line(0,-1){\dfactor}} 
   \advance \grcalcb by -\factor
   \put(\grcalca,\grcalcb) {\line(0,1){\dfactor}} 
   \advance \grcalca by -\hfactor
   \advance \grcalca by \dfactor
   \advance \grcalcb by \dfactor
   \put(\grcalca,\grcalcb) {\line(1,0){\factor}} 
   \advance \grcalcb by \factor
   \advance \grcalcb by -\sfactor
   \put(\grcalca,\grcalcb) {\line(1,0){\factor}} 
   \grcalcc = \factor
   \advance \grcalcc by -\sfactor
   \put(\grcalca,\grcalcb) {\line(0,-1){\grcalcc}} 
   \advance \grcolumn by 1}

 \newcommand{\glmpt}{    
   \grcalca = \grcolumn
   \multiply \grcalca by \factor
   \advance \grcalca by \hfactor
   \grcalcb = \grrow
   \multiply \grcalcb by \factor
   \put(\grcalca,\grcalcb) {\line(0,-1){\dfactor}} 
   \advance \grcalca by -\hfactor
   \advance \grcalca by \dfactor
   \advance \grcalcb by -\dfactor
   \put(\grcalca,\grcalcb) {\line(1,0){\factor}} 
   \advance \grcalcb by -\factor
   \advance \grcalcb by \sfactor
   \put(\grcalca,\grcalcb) {\line(1,0){\factor}} 
   \grcalcc = \factor
   \advance \grcalcc by -\sfactor
   \put(\grcalca,\grcalcb) {\line(0,1){\grcalcc}} 
   \advance \grcolumn by 1}

 \newcommand{\glmpb}{    
   \grcalca = \grcolumn
   \multiply \grcalca by \factor
   \advance \grcalca by \hfactor
   \grcalcb = \grrow
   \multiply \grcalcb by \factor
   \advance \grcalcb by -\factor
   \put(\grcalca,\grcalcb) {\line(0,1){\dfactor}} 
   \advance \grcalca by -\hfactor
   \advance \grcalca by \dfactor
   \advance \grcalcb by \dfactor
   \put(\grcalca,\grcalcb) {\line(1,0){\factor}} 
   \advance \grcalcb by \factor
   \advance \grcalcb by -\sfactor
   \put(\grcalca,\grcalcb) {\line(1,0){\factor}} 
   \grcalcc = \factor
   \advance \grcalcc by -\sfactor
   \put(\grcalca,\grcalcb) {\line(0,-1){\grcalcc}} 
   \advance \grcolumn by 1}

 \newcommand{\glmp}{    
   \grcalca = \grcolumn
   \multiply \grcalca by \factor
   \advance \grcalca by \dfactor
   \grcalcb = \grrow
   \multiply \grcalcb by \factor
   \advance \grcalcb by -\dfactor
   \put(\grcalca,\grcalcb) {\line(1,0){\factor}} 
   \advance \grcalcb by -\factor
   \advance \grcalcb by \sfactor
   \put(\grcalca,\grcalcb) {\line(1,0){\factor}} 
   \grcalcc = \factor
   \advance \grcalcc by -\sfactor
   \put(\grcalca,\grcalcb) {\line(0,1){\grcalcc}} 
   \advance \grcolumn by 1}

 \newcommand{\gcmptb}{    
   \grcalca = \grcolumn
   \multiply \grcalca by \factor
   \advance \grcalca by \hfactor
   \grcalcb = \grrow
   \multiply \grcalcb by \factor
   \put(\grcalca,\grcalcb) {\line(0,-1){\dfactor}} 
   \advance \grcalcb by -\factor
   \put(\grcalca,\grcalcb) {\line(0,1){\dfactor}} 
   \advance \grcalca by -\hfactor
   \advance \grcalcb by \dfactor
   \put(\grcalca,\grcalcb) {\line(1,0){\factor}} 
   \advance \grcalcb by \factor
   \advance \grcalcb by -\sfactor
   \put(\grcalca,\grcalcb) {\line(1,0){\factor}} 
   \advance \grcolumn by 1}

 \newcommand{\gcmpt}{    
   \grcalca = \grcolumn
   \multiply \grcalca by \factor
   \advance \grcalca by \hfactor
   \grcalcb = \grrow
   \multiply \grcalcb by \factor
   \put(\grcalca,\grcalcb) {\line(0,-1){\dfactor}} 
   \advance \grcalcb by -\factor
   \advance \grcalca by -\hfactor
   \advance \grcalcb by \dfactor
   \put(\grcalca,\grcalcb) {\line(1,0){\factor}} 
   \advance \grcalcb by \factor
   \advance \grcalcb by -\sfactor
   \put(\grcalca,\grcalcb) {\line(1,0){\factor}} 
   \advance \grcolumn by 1}

 \newcommand{\gcmpb}{    
   \grcalca = \grcolumn
   \multiply \grcalca by \factor
   \advance \grcalca by \hfactor
   \grcalcb = \grrow
   \multiply \grcalcb by \factor
   \advance \grcalcb by -\factor
   \put(\grcalca,\grcalcb) {\line(0,1){\dfactor}} 
   \advance \grcalca by -\hfactor
   \advance \grcalcb by \dfactor
   \put(\grcalca,\grcalcb) {\line(1,0){\factor}} 
   \advance \grcalcb by \factor
   \advance \grcalcb by -\sfactor
   \put(\grcalca,\grcalcb) {\line(1,0){\factor}} 
   \advance \grcolumn by 1}

 \newcommand{\gcmp}{    
   \grcalca = \grcolumn
   \multiply \grcalca by \factor
   \grcalcb = \grrow
   \multiply \grcalcb by \factor
   \advance \grcalcb by -\factor
   \advance \grcalcb by \dfactor
   \put(\grcalca,\grcalcb) {\line(1,0){\factor}} 
   \advance \grcalcb by \factor
   \advance \grcalcb by -\sfactor
   \put(\grcalca,\grcalcb) {\line(1,0){\factor}} 
   \advance \grcolumn by 1}

 \newcommand{\grmptb}{    
   \grcalca = \grcolumn
   \multiply \grcalca by \factor
   \advance \grcalca by \hfactor
   \grcalcb = \grrow
   \multiply \grcalcb by \factor
   \put(\grcalca,\grcalcb) {\line(0,-1){\dfactor}} 
   \advance \grcalcb by -\factor
   \put(\grcalca,\grcalcb) {\line(0,1){\dfactor}} 
   \advance \grcalca by \hfactor
   \advance \grcalca by -\dfactor
   \advance \grcalcb by \dfactor
   \put(\grcalca,\grcalcb) {\line(-1,0){\factor}} 
   \advance \grcalcb by \factor
   \advance \grcalcb by -\sfactor
   \put(\grcalca,\grcalcb) {\line(-1,0){\factor}} 
   \grcalcc = \factor
   \advance \grcalcc by -\sfactor
   \put(\grcalca,\grcalcb) {\line(0,-1){\grcalcc}} 
   \advance \grcolumn by 1}

 \newcommand{\grmpt}{    
   \grcalca = \grcolumn
   \multiply \grcalca by \factor
   \advance \grcalca by \hfactor
   \grcalcb = \grrow
   \multiply \grcalcb by \factor
   \put(\grcalca,\grcalcb) {\line(0,-1){\dfactor}} 
   \advance \grcalca by \hfactor
   \advance \grcalca by -\dfactor
   \advance \grcalcb by -\dfactor
   \put(\grcalca,\grcalcb) {\line(-1,0){\factor}} 
   \advance \grcalcb by -\factor
   \advance \grcalcb by \sfactor
   \put(\grcalca,\grcalcb) {\line(-1,0){\factor}} 
   \grcalcc = \factor
   \advance \grcalcc by -\sfactor
   \put(\grcalca,\grcalcb) {\line(0,1){\grcalcc}} 
   \advance \grcolumn by 1}

 \newcommand{\grmpb}{    
   \grcalca = \grcolumn
   \multiply \grcalca by \factor
   \advance \grcalca by \hfactor
   \grcalcb = \grrow
   \multiply \grcalcb by \factor
   \advance \grcalcb by -\factor
   \put(\grcalca,\grcalcb) {\line(0,1){\dfactor}} 
   \advance \grcalca by \hfactor
   \advance \grcalca by -\dfactor
   \advance \grcalcb by \dfactor
   \put(\grcalca,\grcalcb) {\line(-1,0){\factor}} 
   \advance \grcalcb by \factor
   \advance \grcalcb by -\sfactor
   \put(\grcalca,\grcalcb) {\line(-1,0){\factor}} 
   \grcalcc = \factor
   \advance \grcalcc by -\sfactor
   \put(\grcalca,\grcalcb) {\line(0,-1){\grcalcc}} 
   \advance \grcolumn by 1}

 \newcommand{\grmp}{    
   \grcalca = \grcolumn
   \multiply \grcalca by \factor
   \advance \grcalca by \factor
   \advance \grcalca by -\dfactor
   \grcalcb = \grrow
   \multiply \grcalcb by \factor
   \advance \grcalcb by -\dfactor
   \put(\grcalca,\grcalcb) {\line(-1,0){\factor}} 
   \advance \grcalcb by -\factor
   \advance \grcalcb by \sfactor
   \put(\grcalca,\grcalcb) {\line(-1,0){\factor}} 
   \grcalcc = \factor
   \advance \grcalcc by -\sfactor
   \put(\grcalca,\grcalcb) {\line(0,1){\grcalcc}} 
   \advance \grcolumn by 1}

 \newcommand{\gwmuh}[3]{    
   \grcalca = \grcolumn
   \multiply \grcalca by \factor
   \grcalcb = #2
   \advance \grcalcb by #3
   \multiply \grcalcb by \qfactor
   \advance \grcalca by \grcalcb
   \grcalcb = \grrow
   \multiply \grcalcb by \factor
   \grcalcc = #3
   \advance \grcalcc by -#2
   \multiply \grcalcc by \hfactor
   \grcalcd = \factor
   \advance \grcalcd by \hfactor
   \put(\grcalca,\grcalcb){\oval(\grcalcc,\grcalcd)[b]}
   \grcalca = \grcolumn
   \multiply \grcalca by \factor
   \grcalcc = #1
   \multiply \grcalcc by \hfactor
   \advance \grcalca by \grcalcc
   \advance \grcalcb by -\hfactor
   \advance \grcalcb by -\qfactor
   \put(\grcalca,\grcalcb) {\line(0,-1){\qfactor}} 
   \advance \grcolumn by #1}

 \newcommand{\gwcmh}[3]{   
   \grcalca = \grcolumn
   \multiply \grcalca by \factor
   \grcalcb = #2
   \advance \grcalcb by #3
   \multiply \grcalcb by \qfactor
   \advance \grcalca by \grcalcb
   \grcalcb = \grrow
   \advance \grcalcb by -1
   \multiply \grcalcb by \factor
   \grcalcc = #3
   \advance \grcalcc by -#2
   \multiply \grcalcc by \hfactor
   \grcalcd = \factor
   \advance \grcalcd by \hfactor
   \put(\grcalca,\grcalcb){\oval(\grcalcc,\grcalcd)[t]}
   \grcalca = \grcolumn
   \multiply \grcalca by \factor
   \grcalcc = #1
   \multiply \grcalcc by \hfactor
   \advance \grcalca by \grcalcc
   \advance \grcalcb by \factor
   \put(\grcalca,\grcalcb) {\line(0,-1){\qfactor}} 
   \advance \grcolumn by #1}

 \newcommand{\gsbox}[1]{
   \grcalca = \grcolumn
   \multiply \grcalca by \factor
   \grcalcb = \grrow
   \multiply \grcalcb by \factor
   \advance \grcalcb by -\factor
   \grcalcc = #1
   \multiply \grcalcc by \factor
   \grcalcd = \factor
   \put(\grcalca,\grcalcb){\framebox(\grcalcc,\grcalcd){}}}

\makeatletter																																													
\DeclareMathSizes{\@xpt}{\@xpt}{6}{5}																																	
\makeatother																																													
\makeatletter
\def\namedlabel#1#2{\begingroup
    #2%
    \def\@currentlabel{#2}%
    \phantomsection\label{#1}\endgroup
}
\makeatother

\allowdisplaybreaks

\newenvironment{invisible}[1][\unskip]{
	\noindent
	\color{red}
	[{\textbf{\color{blue}TBH}: \textit{#1}}
}{]}

\newenvironment{personal}{{\noindent \textbf{\color{green}Personal notes (To Be Hidden)}:\quad}\color{green}}{}

\theoremstyle{plain}

\newtheorem{theorem}{Theorem}[section]
\newtheorem{lemma}[theorem]{Lemma}
\newtheorem{proposition}[theorem]{Proposition}
\newtheorem{corollary}[theorem]{Corollary}

\newtheorem{lemma*}[theorem]{Lemma}
\newtheorem{proposition*}[theorem]{Proposition}

\theoremstyle{definition}

\newtheorem{definition}[theorem]{Definition}

\newtheorem{remark}[theorem]{Remark}
\newtheorem{notation}[theorem]{Notation}

\newcommand{\K}{\Bbbk}

\newcommand{\varfun}[3]{#1 \colon #2 \to #3}

\newcommand{\coinv}[2]{{#1}^{\mathrm{co}{#2}}} 

\newcommand{\mf}[1]{\mathfrak{#1}} 
\newcommand{\mb}[1]{\boldsymbol{#1}} 
\newcommand{\micro}[1]{{\scriptscriptstyle  #1}}

\newcommand{\forget}{\mb{\upomega}} 
\newcommand{\frho}{\delta} 
\newcommand{\Nat}{\mathsf{Nat}} 
\newcommand{\coend}{\underline{\mathrm{coend}}} 
\renewcommand{\ker}{\mathsf{ker}} 

\newcommand{\coker}{\mathsf{coker}} 
\newcommand{\id}{\mathsf{Id}} 

\newcommand{\Hom}[6]{{_{#1}^{#2}\mathsf{Hom}_{#3}^{#4}}\left({#5},{#6}\right)} 
\newcommand{\Homk}{\mathsf{Hom}} 
\newcommand{\End}[2]{\mathsf{End}_{#1}(#2)} 

\newcommand{\fvect}{{\M_{f}}} 
\newcommand{\Set}{{\mathsf{Set}}} 

\newcommand{\cC}{{\mathcal C}}
\newcommand{\cD}{{\mathcal D}}
\newcommand{\cE}{{\mathcal E}}
\newcommand{\cF}{{\mathcal F}}
\newcommand{\cG}{{\mathcal G}}

\newcommand{\cI}{{\mathcal I}}

\newcommand{\cL}{{\mathcal L}}
\newcommand{\cM}{{\mathcal M}}

\newcommand{\cP}{{\mathcal P}}

\newcommand{\cR}{{\mathcal R}}

\newcommand{\cU}{{\mathcal U}}
\newcommand{\cV}{{\mathcal V}}

\newcommand{\M}{\mathfrak{M}} 
\newcommand{\I}{\mathbb{I}} 
\newcommand{\calpha}{\mf{a}}
\newcommand{\clambda}{\mf{l}}
\newcommand{\crho}{\mf{r}}
\newcommand{\ev}[1]{\mathsf{ev}_{{{#1}}}} 
\newcommand{\db}[1]{\mathsf{db}_{{{#1}}}} 

\newcommand{\ie}{i.e.~}
\newcommand{\eg}{e.g.~}

\CompileMatrices

\excludeversion{invisible} 
\excludeversion{personal}

\title{Coquasi-bialgebras with preantipode and rigid monoidal categories}
\author{Paolo Saracco}
\address{D\'epartement de Math\'ematique, Universit\'e Libre de Bruxelles, Boulevard du Triomphe, B-1050 Brussels, Belgium.}
\thanks{This paper was written while the author was member of the "National Group for Algebraic and Geometric Structures and their Applications" (GNSAGA-INdAM). The author is sincerely grateful to Alessandro Ardizzoni and Claudia Menini for their contribution.}
\urladdr{sites.google.com/site/paolosaracco}
\email{paolo.saracco@ulb.ac.be}

\begin{document}

\begin{abstract}
By a theorem of Majid, every monoidal category with a neutral quasi-monoidal functor to finitely-generated and projective $\K$-modules gives rise to a coquasi-bialgebra. We prove that if the category is also rigid, then the associated coquasi-bialgebra admits a preantipode, providing in this way an analogue for coquasi-bialgebras of Ulbrich's reconstruction theorem for Hopf algebras. When $\K$ is field, this allows us to characterize coquasi-Hopf algebras as well in terms of rigidity of finite-dimensional corepresentations.
\end{abstract}

\maketitle


\begin{invisible}
\textbf{This version is a new starting point after rejection from Canadian. It may not contain the previous invisible sections and it contains some minor corrections with respect to the previous versions. Refer to them for the old invisible stuff.}
\end{invisible}

\section*{Introduction}

A well-known result in the theory of Hopf algebras states that one can reconstruct, in a suitable way, a Hopf algebra from its category of finite-dimensional corepresentations. In details, if $\cC $ is a $\Bbbk$-linear, abelian, rigid symmetric monoidal category which is essentially small, and if $\forget:\cC\rightarrow \fvect $ is a $\Bbbk $-linear, exact, faithful, monoidal functor, then there exists a commutative Hopf algebra $H$, unique up to isomorphism, such that $\forget$ factorizes through an equivalence of categories $\forget^{H}:\cC \rightarrow {\M_f ^{H}}$ followed by the forgetful functor; in fact $H$ represents the functor 
\begin{equation*}
R\rightarrow \mathrm{Aut}^{\otimes}(\forget\otimes R)
\end{equation*}
which associates any commutative $\Bbbk$-algebra $R$ with the group of monoidal natural automorphisms of $\forget\otimes R:\cC\rightarrow \mathrm{Mod}_R$ sending $X$ to $\forget(X)\otimes R$, see \cite{Riv},  \cite{DeligneTannaka} and \cite{JoyalStreet}. In particular, if $\cC $ is already the category of finite-dimensional right comodules over a commutative Hopf algebra $A$, then one can show that $A\cong H$ as Hopf algebras. In \cite{Ul}, Ulbrich showed that even in case the symmetry condition is dropped, it is still possible to construct an associated Hopf algebra $H$.

In \cite{Majid-reconstruction}, Majid extended this result to coquasi-bialgebras (or dual quasi-bialgebras), proving that if $\cC $ is an essentially small monoidal category endowed with a functor $\forget:\cC \rightarrow \fvect $ that respects the tensor product in a suitable way (but that is not necessarily monoidal), then there is a coquasi-bialgebra $H$ such that $\forget$ factorizes through a monoidal functor $\forget^{H}:\cC  \rightarrow {^{H}\fvect } $ followed by the forgetful functor.

In \cite{Ardi-Pava}, Ardizzoni and Pavarin introduced preantipodes to characterize those coquasi-bialgebras over a field for which a (suitable) structure theorem for coquasi-Hopf bicomodules holds and in \cite[Theorem 2.6]{Sch-TwoChar} Schauenburg proved (in a non-constructive way) that preantipodes characterize also those coquasi-bialgebras whose category of finite-dimensional comodules is rigid.

Inspired by these results, we are going to show in Section \ref{sec:tannaka} that if $\cC$ is an essentially small right rigid monoidal category together with a quasi-monoidal functor $\forget:\cC \rightarrow \fvect $ to the category of finitely-generated and projective $\K$-modules, then there exists a preantipode for the coendomorphism coquasi-bialgebra $H$ of $\forget$ (Proposition \ref{lemma:pre1}). In particular, this will allow us to reconstruct a coquasi-bialgebra with preantipode from its category of finite-dimensional left comodules, in the spirit of the classical Tannaka-Krein duality. Our approach presents three remarkable advantages. First of all, nowhere we will assume to have an isomorphism between the ``underlying'' $\K$-module $\forget(X^\star)$ of a dual object and the dual $\K$-module $\forget(X)^*$ of the ``underlying'' object (as it is done for example in \cite[\S3]{reconQuasiHopf}). In fact, we will see that if that is the case, then $H$ can be equipped with a coquasi-Hopf algebra structure. Secondly, we will develop our main construction working over a generic commutative ring $\K$ instead of over a field. Thirdly, we will not only show that a preantipode exists, but we will show how to construct it explicitly. 

Then, we will apply this result to show how to recover properties such as the uniqueness of preantipodes and the fact that any coquasi-bialgebra morphism automatically preserves them from their categorical counterparts. We will also recover the characterization of coquasi-bialgebras with preantipode as those coquasi-bialgebras whose category of finite-dimensional corepresentations is rigid and we will characterize coquasi-Hopf algebras as those for which in addition $\forget(-^\star)$ and $\forget(-)^*$ are isomorphic (Theorem \ref{thm:characterization}). In conclusion, we will see in \S\ref{sec:finitedual} how we can endow the finite dual coalgebra of a quasi-bialgebra with preantipode with a structure of coquasi-bialgebra with preantipode.

\section{Coquasi-bialgebras and preantipodes}

We extend here the notion of preantipode as it has been introduced in \cite{Ardi-Pava} to the case of coquasi-bialgebras over a commutative ring. We prove that when it exists, it has to be unique and that coquasi-bialgebra morphisms have to preserve it \emph{under the additional assumption that the coquasi-bialgebras are $\K$-flat}.

\subsection{Monoidal categories and coquasi-bialgebras}\label{Subsec:moncat}

A \emph{monoidal category} is a category $\cC $ endowed with a functor $\otimes :\cC \times \cC \rightarrow \cC $, called the \emph{tensor product}, with a distinguished object $\I $, called the \emph{unit}, and with three natural isomorphisms 
\begin{equation*}
\arraycolsep=10pt\def\arraystretch{1.5}
\begin{array}{cc}
\calpha :\otimes \,(\otimes \times \id _{\cC })\rightarrow \otimes  \,(\id _{\cC }\times \otimes ) & \text{(\emph{associativity constraint})} \\
\clambda :\otimes \, (\I \times \id _{\cC })\rightarrow \id _{\cC }, \quad \crho :\otimes \, (\id _{\cC }\times \I )\rightarrow \id _{\cC } & \text{(\emph{left and right unit constraints})}
\end{array}
\end{equation*}
that satisfy the \emph{Pentagon} and the \emph{Triangle Axioms}, that is, for all $X,Y,Z,W$ in $\cC$
\begin{gather*}
\left(X\otimes \calpha_{Y,Z,W}\right) \circ \calpha_{X,Y \otimes Z,W} \circ \left( \calpha_{X,Y,Z}\otimes W\right) = \calpha_{X,Y,Z\otimes W}\circ\calpha_{X\otimes Y,Z,W}, \\
\left( X\otimes \clambda_Y\right) \circ \calpha_{X,\I,Y} = \crho_X\otimes Y.
\end{gather*}
A \emph{quasi-monoidal functor} between $\left( \cC ,\otimes ,\I ,\calpha
,\clambda ,\crho \right) $ and $\left( \cC ^{\prime },\otimes^{\prime} ,\I^{\prime}
,\calpha ^{\prime },\clambda ^{\prime },\crho ^{\prime }\right) $ is a functor $
\forget:\cC \rightarrow \cC ^{\prime }$ together with
an isomorphism $\varphi _{0}:\I^{\prime}\rightarrow \forget\left( 
\I \right) $ and a natural isomorphism $\varphi=\left(\varphi _{X,Y}:\forget
\left( X\right) \otimes^{\prime} \forget\left( Y\right) \rightarrow 
\forget\left( X\otimes Y\right)\right)_{X,Y\in\cC} $ in $\cC ^{\prime }$. Omitting the composition symbols, a quasi-monoidal functor $\forget$ is said to be \emph{neutral} if
\begin{equation}\label{eq:unitfunctor}
\forget\left( \clambda_X \right)  \varphi_{\I,X}  \left( \varphi_0\otimes^{\prime} \forget(X)\right) = \clambda^{\prime}_{\forget(X)}, \quad \forget\left( \crho_X \right)  \varphi_{X,\I}  \left( \forget(X) \otimes^{\prime} \varphi_0 \right) = \crho^{\prime}_{\forget(X)}
\end{equation}
for all $X$ in $\cC$. Furthermore, $\forget$ is said to be \emph{monoidal}\footnote{In \cite[Definition 3.5]{Aguiar}, these are called \emph{strong} monoidal functors.} if 
\begin{equation}\label{eq:assfunctor}
\forget ( \calpha_{X,Y,Z}  )  \varphi_{X\otimes Y,Z}   ( \varphi_{X,Y} \otimes^{\prime} \forget(Z)  ) = \varphi_{X,Y\otimes Z}   (\forget(X) \otimes^{\prime} \varphi_{Y,Z}  )  \calpha^{\prime}_{\forget(X), \forget(Y), \forget(Z)}
\end{equation}
for all $X,Y,Z$ in $\cC$. It is said to be \emph{strict} if $\varphi_0$ and $\varphi$ are the identities.

The notions of (co)algebra and (co)module over a (co)algebra can be introduced in the general setting of monoidal categories (see e.g. \cite[\S1.2]{Aguiar}, where (co)algebras are called (co)monoids). Given an algebra $A$ in $\cC $, one can define the categories $_{A} \cC $, $\cC _{A}$ and $_{A}\cC _{A}$ of left, right and two-sided modules over $A$, respectively. Similarly, given a coalgebra $C$ in  $\cC $, one can define the categories of $C$-comodules $^{C}\mathcal{C },\cC ^{C},{^{C}\cC ^{C}}$.

\begin{invisible}
Caveat: if $N\in{^H_{}\M^H_H}$ there is no easy way to prove that $\coinv{N}{H}\in{^H\M}$ if the tensor product $H\otimes-$ does not preserves equalizers! That's why my theory in the quasi case works pretty well even for $\K$ a commutative ring, but Ardi's one doesn't. 
\end{invisible}

Henceforth and unless stated otherwise, we will fix a base commutative ring $\Bbbk$ and we will assume to work in the monoidal category $ \M $ of $\Bbbk $-modules: all (co)algebras will be $\Bbbk $-(co)algebras, the unadorned tensor product $\otimes $ will denote the tensor product over $\Bbbk $ and $\Homk(V,W)$ will be the set of $\Bbbk$-linear morphisms from $V$ to $W$. We will often omit the composition symbols between maps as we did above. In order to deal with comultiplications and coactions, we will use the following variation of \emph{Sweedler's Sigma Notation} (cf. \cite[\S1.2]{Swe})
\begin{equation*}
\Delta (x):=\sum x_{1}\otimes x_{2},\qquad \rho_V^{r}(v):=\sum v_{0}\otimes v_{1}, \qquad \rho_W^{l}(w):=\sum w_{-1}\otimes w_{0}
\end{equation*}
for every coalgebra $C$, right $C$-comodule $V$, left $C$-comodule $W$ and for all $x\in C$, $v\in V$ and $w\in W$. Recall that
\begin{itemize}
\item if $W$ is a left $C$-comodule, then its linear dual $W^\ast:=\Homk(W,\Bbbk)$ is naturally a right $C$-comodule with $\sum f_0\otimes f_1$ uniquely determined by $\sum f_0(w)f_1=\sum w_{-1}f\left(w_0\right)$ for all $w\in W$;
\item if $A$ is an algebra, then $\Homk(C,A)$ is a monoid with composition law defined by $(f*g)(x)=\sum f(x_1)g(x_2)$ for all $f,g\in\Homk(C,A)$ and $x\in C$ (the \emph{convolution product}) and if $M$ is an $A$-bimodule then we may as well consider $(f*\phi*g)(x)=\sum f(x_1)\cdot \phi(x_2)\cdot g(x_3)$ for all $f,g\in\Homk(C,A)$, $\phi\in\Homk(C,M)$ and $x\in C$.
\end{itemize}

The following result is formally dual to \cite[Theorem 1]{ArdiBulacuMenini} and has already been mentioned in \cite[\S2.3]{Schauenburg-Kac}. Since the proof is quite long, technical and not of particular interest, it is omitted.

\begin{proposition}\label{prop:coquasi}
For a coalgebra $(C,\Delta,\varepsilon)$ there is a bijective correspondence between
\begin{itemize}
\item monoidal structures on ${^C\M}$ such that the underlying functor $\cU:{^C\M}\to\M$ is quasi-monoidal;
\item sets of morphisms $\left\{m,u,\omega,l,r\right\}$ such that $\omega :C\otimes
C\otimes C\rightarrow \Bbbk $, $l,r:C\to\K$ are convolution invertible linear maps, $m:C\otimes C \to C,u:\K\to C$ are coalgebra morphisms and
\begin{gather}
\omega \left( C\otimes C\otimes m\right) \ast \omega \left( m\otimes
C\otimes C\right) =\left( \varepsilon \otimes \omega \right) \ast \omega
\left( C\otimes m\otimes C\right) \ast \left( \omega \otimes \varepsilon
\right) ,  \label{eq:3-cocycle} \\
\omega \left( C\otimes u\otimes C\right) =r^{-1} \otimes l , \quad m\left( u\otimes C\right)*l = l*C , \quad m\left( C\otimes u\right)*r = r*C, \label{eq:quasi-unitairity cocycle} \\
m\left( C\otimes m\right) \ast \omega =\omega \ast m\left( m\otimes C\right). \notag 
\end{gather}
\end{itemize}
\end{proposition}

A \emph{coquasi-bialgebra} (or \emph{dual quasi-bialgebra}) is a coassociative and counital coalgebra $(H,\Delta ,\varepsilon )$ endowed with a multiplication $m:H\otimes H\rightarrow H$, a unit $u:\Bbbk \rightarrow H$ and three linear maps $\omega :H\otimes H\otimes H\rightarrow \Bbbk $, $l,r:H\to\K$ such that the conditions of Proposition \ref{prop:coquasi} are satisfied. We refer to $\omega $ as the \emph{reassociator} of the coquasi-bialgebra. A \emph{morphism of coquasi-bialgebras} 
\begin{equation*}
f:\left( H,m,u,\Delta ,\varepsilon ,\omega , l ,r\right) \rightarrow \left(H^{\prime },m^{\prime },u^{\prime },\Delta ^{\prime },\varepsilon ^{\prime},\omega ^{\prime }, l', r'\right)
\end{equation*}
is a coalgebra homomorphism $f:\left( H,\Delta ,\varepsilon \right)\rightarrow \left( H^{\prime },\Delta ^{\prime },\varepsilon ^{\prime}\right) $ such that 
\begin{equation*}
m^{\prime }\, (f\otimes f)=f\, m,\qquad f\, u=u^{\prime },\qquad\omega ^{\prime }\, \left( f\otimes f\otimes f\right) =\omega, \qquad l'f=l,\qquad r'f=r .
\end{equation*}
In particular, the category ${^{H} \M }$ of left comodules over a coquasi-bialgebra $H$ comes endowed with a monoidal structure such that the underlying functor $\cU:{^{H} \M }\rightarrow  \M $ is a strict quasi-monoidal functor. Explicitly, given two left $H$-comodules $V$ and $W$, their tensor product $V\otimes W$ is an $H$-comodule via the diagonal coaction $\rho_{ {V\otimes W}}\left( v\otimes w\right) =\sum v_{-1}w_{-1}\otimes v_{0}\otimes w_{0}.$ The unit is $\Bbbk$, regarded as a left $H$-comodule via the trivial
coaction $\rho_{\Bbbk} \left( k\right) = 1_{H}\otimes k$. The constraints are given by 
\begin{gather*}
\calpha _{ {U,V,W}}(u\otimes v\otimes w):=\sum \omega^{-1}(u_{-1}\otimes v_{-1}\otimes w_{-1})u_{0}\otimes v_{0}\otimes w_{0}, \\
{\clambda_{V}}(1\otimes v)=\sum l(v_{-1})v_0 \qquad {\crho_V}(v\otimes 1)=\sum r(v_{-1})v_0
\end{gather*}
for every $U,V,W\in {^{H} \M }$ and all $ u\in U,v\in V,w\in W$. Moreover, every morphism of coquasi-bialgebras $f:H\rightarrow H^{\prime }$ induces a strict monoidal functor ${^{f} \M }: {^{H} \M }\rightarrow {^{H^{\prime}} \M }$, which is given by the assignments 
\begin{equation*}
\begin{split}
{^{f} \M }\left( X,\rho _{ {X}}:X\rightarrow H\otimes X\right) =\left( X,\left(f\otimes X\right) \rho _{ {X}}\right), \qquad {^{f} \M }\left(\gamma :X\rightarrow Y\right) =\gamma.
\end{split}
\end{equation*}
Dualizing \cite{Dri}, one can check that any coquasi-bialgebra is equivalent to one in which $l=\varepsilon=r$. For this reason, we will focus only on the latter case and from now on all coquasi-bialgebras will satisfy
\begin{equation}\tag{$4'$}\label{eq:quasi-unitairity cocycle2}
\omega \left( C\otimes u\otimes C\right) =\varepsilon\otimes \varepsilon , \qquad m\left( u\otimes C\right) = C = m\left( C\otimes u\right)
\end{equation}
instead of relations \eqref{eq:quasi-unitairity cocycle}. Moreover, all quasi-monoidal functors will be neutral and hence we will omit to specify it.

Dually to coquasi-bialgebras we have \emph{quasi-bialgebras}, that is to say, ordinary algebras $A$ with a counital comultiplication which is coassociative up to conjugation by a suitable invertible element $\Phi\in A\otimes A\otimes A$. 

\subsection{Preantipodes for coquasi-bialgebras}\label{section4}
The following definition traces word by word \cite[Definition 3.6]{Ardi-Pava}.

\begin{definition}
\label{def:preantipode} A \emph{preantipode} for a coquasi-bialgebra $H$ is a $\Bbbk $-linear endomorphism $S:H\rightarrow H$ such that, for all $h\in H$, 
\begin{gather} 
\sum S(h_{1})_{1}h_{2}\otimes S(h_{1})_{2}=1_{H}\otimes S(h), \notag \\
\sum S(h_{2})_{1}\otimes h_{1}S(h_{2})_{2}=S(h)\otimes 1_{H}, \notag \\
\sum \omega (h_{1}\otimes S(h_{2})\otimes h_{3})=\varepsilon (h).\label{fond S}
\end{gather}
\end{definition}

\begin{remark}\label{Lempreant}
Let $H$ be a coquasi-bialgebra with a preantipode $S$. Then
\begin{equation*}
\sum h_{1}S(h_{2}) =\varepsilon S(h)1_{H}=\sum S(h_{1})h_{2}
\end{equation*}
for all $h\in H$. In particular, if $\varepsilon S(h)=\varepsilon \left( h\right) $ then $S$ is an ordinary antipode.
\end{remark}

A coquasi-bialgebra $H$ turns out to be an algebra in the monoidal category ${^{H} \M ^{H}}$ \cite[\S2]{Sch-TwoChar}. Thus we may consider the so-called category of \emph{right coquasi-Hopf  $H$-bicomodules} ${^{H} \M _{H}^{H}}:=\left({^{H} \M ^{H}}\right)_{H}$.

Assume that $H$ is flat over $\K$. Then the functor $F:{^{H}\M}\rightarrow {^{H}\M_{H}^{H}}$ given by $F(V):=V\otimes H$ admits a right adjoint $G:{^{H} \M _{H}^{H}}\rightarrow {{^{H} \M }}$, $G\left( M\right) :=M^{coH}$, where $M^{\mathrm{co}H}:=\{m\in M\mid m_{0}\otimes
m_{1}=m\otimes 1_{H}\}$ is the space of right $H$-coinvariant elements in $M$. The counit $\epsilon:FG\rightarrow \mathrm{id}$ and the unit $\eta :\mathrm{id}\rightarrow GF$ of the adjunction are given respectively by  $\epsilon _{ {M}}(x\otimes h):=xh$ and $\eta _{ {N}}\left(n\right) :=n\otimes 1_{H}$ for every $M\in $ $^{H} \M _{H}^{H}$, $N\in {^{H} \M }$ and for all $m\in M$, $n\in N$, $h\in H$ (we refer to \cite{Ardi-Pava} for details).
Then one can mimic step by step the proof of \cite[Theorem 3.9]{Ardi-Pava} to prove the following.

\begin{theorem}\label{Teoequidual}
Under the standing assumption that $H$ is $\K$-flat, the adjunction $(F,G)$ is an equivalence of categories if and only if $H$ admits a preantipode.
\end{theorem}

We might have given now a direct proof of the fact that coquasi-bialgebra morphisms preserve preantipodes, but we opted for a less direct approach relying on the subsequent proposition suggested by Alessandro Ardizzoni. The effort is the same and we think that the general result in Proposition \ref{prop:preantipodemorph} deserves to be highlighted, as it may find applications in other contexts.

\begin{proposition}\label{prop:preantipodemorph}
Let $(C,\Delta _{C},\varepsilon _{C})$ be a coalgebra and let $(H,\Delta_{H},\varepsilon _{H},m,u,\omega ,S)$ be a $\K$-flat coquasi-bialgebra with preantipode $S$. Assume that $g,h:C\rightarrow H$ are $\Bbbk $-linear maps such that $g$ is a coalgebra morphism and $g$ and $h$ satisfy:
\begin{gather}
\sum h(z_{2})_{1}\otimes g(z_{1})h(z_{2})_{2}=h(z)\otimes 1, \label{eq:preantmorph1}\\ 
\sum h(z_{1})_{1}g(z_{2})\otimes h(z_{1})_{2}=1\otimes h(z), \label{eq:preantmorph2}\\ 
\sum \omega (g(z_{1})\otimes h(z_{2})\otimes g(z_{3}))=\varepsilon (z), \label{eq:preantmorph3}
\end{gather}
for all $z\in C$. Then $h=S g$.
\end{proposition}

\begin{proof}
As in \cite[\S3.5]{Ardi-Pava}, consider the coquasi-Hopf bicomodule $H \hat{\otimes }H:=H\otimes H$ with explicit structures given by
\begin{gather*}
\rho ^{r}\left( x\otimes y\right) =\sum x_{1}\otimes y_{1}\otimes x_{2}y_{2}, \qquad \rho ^{l}\left( x\otimes y\right) =\sum y_{1}\otimes x\otimes y_{2}, \\
\left( x\otimes y\right) h =\sum x_{1}\otimes y_{1}h_{1}\omega \left( x_{2}\otimes y_{2}\otimes h_{2}\right),
\end{gather*}
for all $x,y,h\in H$. Consider also the distinguished component $\hat{\epsilon }:\left( H\hat{ \otimes }H\right) ^{\mathrm{co}H}\otimes H\rightarrow H\hat{\otimes }H$ of the counit of the adjunction $(F,G)$, which is given explicitly by $\hat{\epsilon }\left( x\otimes y\otimes h\right)=\sum x_{1}\otimes y_{1}h_{1}\omega \left( x_{2}\otimes y_{2}\otimes h_{2}\right) $. Since $H$ admits a preantipode, it is invertible with inverse $\hat{\epsilon }^{-1}\left( x\otimes y\right) = \sum \left(\left( x_{1}\otimes S\left( x_{2}\right) \right) \otimes x_{3}\right) y$ for all $x,y,h\in H$. Finally, consider the assignment $\beta:C\rightarrow H\otimes H\otimes H$ given by 
\begin{equation*}
\beta (z)=\sum g(z_{1})\otimes h(z_{2})\otimes g(z_{3})
\end{equation*}
for all $z\in C$. Observe that 
\begin{align*}
\rho ^{r} & \left(\sum g(z_{1})\otimes h(z_{2})\right) =\sum g(z_{1})_{1}\otimes
h(z_{2})_{1}\otimes g(z_{1})_{2}h(z_{2})_{2} \\
&\stackrel{\left( *\right) }{=}\sum g(z_{1})\otimes h(z_{3})_{1}\otimes g(z_{2})h(z_{3})_{2}\stackrel{\eqref{eq:preantmorph1}}{=}\sum g(z_{1})\otimes h(z_{2})\otimes 1,
\end{align*}
where in $(*)$ we used the hypothesis that $g$ is comultiplicative, whence $\sum g(z_{1})\otimes h(z_{2})\in \left( H\hat{\otimes }H\right) ^{
\mathrm{co}H}$ for all $z\in C$. Therefore for all $z\in C$ we can compute
\begin{align*}
\hat{\epsilon }\beta \left( z\right) &\stackrel{\phantom{(54)}}{=}\hat{\epsilon }
\left(\sum g(z_{1})\otimes h(z_{2})\otimes g(z_{3})\right) \\
 &\stackrel{\phantom{(54)}}{=}\sum g(z_{1})_{1}\otimes h(z_{2})_{1}g(z_{3})_{1}\omega \left(
g(z_{1})_{2}\otimes h(z_{2})_{2}\otimes g(z_{3})_{2}\right) \\
&\stackrel{\phantom{(54)}}{=}\sum g(z_{1})\otimes h(z_{3})_{1}g(z_{4})\omega \left( g(z_{2})\otimes h(z_{3})_{2}
\otimes g(z_{5})\right) \\
&\stackrel{\eqref{eq:preantmorph2}}{=}\sum g(z_{1})\otimes 1_{H}\omega \left( g(z_{2})\otimes h(z_{3})\otimes
g(z_{4})\right) \stackrel{\eqref{eq:preantmorph3}}{=}g(z)\otimes 1_{H}
\end{align*}
so that
\begin{equation*}
\beta \left( z\right) =\hat{\epsilon }^{-1}\left( g(z)\otimes 1_{H}\right) = \sum g(z)_{1}\otimes S\left( g(z)_{2}\right) \otimes g(z)_{3} = \sum g(z_{1})\otimes S\left( g(z_{2})\right) \otimes g(z_{3})
\end{equation*}
and, by applying $\varepsilon \otimes H\otimes \varepsilon $ to both sides,
\begin{align*}
& h(z) = \sum \varepsilon \left( g(z_{1})\right) h(z_{2})\varepsilon \left(
g(z_{3})\right) =  \left( \varepsilon \otimes H\otimes \varepsilon \right)
\left(\sum g(z_{1})\otimes h(z_{2})\otimes g(z_{3})\right) \\
&=\left( \varepsilon \otimes H\otimes \varepsilon \right) (\beta 
\left( z\right) )=\left( \varepsilon \otimes H\otimes \varepsilon \right)
\left(  \sum g(z_{1})\otimes S\left( g(z_{2})\right) \otimes g(z_{3})\right)
=S(g(z)). \qedhere
\end{align*}
\end{proof}

\begin{proposition}
If $(H,S_H),(L,S_L)$ are $\K$-flat coquasi-bialgebras with preantipode and $f:H\to L$ is a morphism of coquasi-bialgebras, then $f S_H = S_L f$. In particular, the preantipode for a $\K$-flat coquasi-bialgebra $H$ is unique.
\end{proposition}

\begin{proof}
Since $f$ is a morphism of coquasi-bialgebras, it is in particular a coalgebra morphism and $f S_{ {A}}=h$ satisfies \eqref{eq:preantmorph1}, \eqref{eq:preantmorph2} and \eqref{eq:preantmorph3} of Proposition \ref{prop:preantipodemorph}, whence $fS_A=S_Bf$.
\begin{invisible}
Indeed, a direct computation
\begin{gather*}
\sum  (gS_{ {A}}(z_{2}) ) _{1}\otimes g(z_{1}) (g S_{ {A}}(z_{2}) )_{2} = \hspace{-1pt} \sum g (S_{ {A}}(z_{2})_{1} )\otimes g ( z_{1}S_{ {A}}(z_{2})_{2} ) \hspace{-1pt} = \hspace{-1pt} gS_{ {A}}(z) \otimes 1, \\ 
\sum ( g S_{ {A}}(z_{1}) ) _{1}g(z_{2})\otimes  (g S_{ {A}}(z_{1}) )_{2} = \hspace{-1pt} \sum g ( S_{ {A}}(z_{1})_{1}z_{2} ) \otimes g ( S_{ {A}}(z_{1})_{2} ) \hspace{-1pt}= \hspace{-1pt} 1\otimes g S_{ {A}}(z), \\ 
\sum \omega _{ {B}}(g(z_{1})\otimes gS_{ {A}}(z_{2}) \otimes g(z_{3})) = \sum \omega_{ {A}} (z_{1}\otimes S_{ {A}}(z_{2})\otimes z_{3}) \hspace{-1pt} = \hspace{-1pt} \varepsilon_{ {A}} (z),
\end{gather*}
shows that \eqref{eq:preantmorph2} and \eqref{eq:preantmorph3} of Proposition \ref{prop:preantipodemorph} are satisfied.
\end{invisible}
Now, assume that $S$ and $T$ are two preantipodes for $H$. The first claim applied to the case $(H,S), (H,T)$ and $f=\id_H$ entails that $S=T$. 
\begin{invisible}
For our own sake, let us retrieve here a direct and ``easy'' proof of the fact that coquasi-bialgebra morphisms commute with preantipodes (as it has been suggested by a referee on a previous version of the paper).

Set ${H\hat{\otimes}H}:=H\otimes H$ with structures
\begin{gather*}
\rho _{ {H\hat{\otimes} H}}^{l}(x\otimes y)=\sum y_{1}\otimes (x\otimes
y_{2}), \qquad \rho _{ {H\hat{\otimes} H}}^{r}(x\otimes y)=\sum (x_1\otimes y_{1})\otimes x_2y_{2}, \\
\mu _{ {H\hat{\otimes} H}}^{r}\left( (x\otimes y)\otimes h\right) =(x\otimes
y)h=\sum \omega ^{-1}(x_{1}\otimes y_{1}\otimes h_{1})x_{2}\otimes
y_{2}h_{2},
\end{gather*}
for every $x,y,h\in H$. Recall from the proof of \cite[Theorem 3.5]{Ardi-Pava} that if a preantipode $S$ exists
then an inverse for the component of the counit $\epsilon _{ {H\hat{\otimes}H}}:\left( H\,\hat{\otimes }\,H\right) ^{\mathrm{co}H}\otimes H\rightarrow H\,\hat{\otimes }\,H$ is explicitly given by the relation $\epsilon _{ {H\hat{\otimes }H}}^{-1}\left( h\otimes 1\right) =\sum
h_{1}\otimes S\left( h_{2}\right) \otimes h_{3}$. Notice that $f:H\to L$ induces a $\K$-linear map $co(f):\coinv{\left({H\hat{\otimes}H}\right)}{H}\to \coinv{\left({L\hat{\otimes}L}\right)}{L}$ given by the (co)restriction of $f\otimes f$.
Indeed, 
\begin{align*}
\rho_{L\hat{\otimes}L}^r\left(\sum_if(x_i)\otimes f(y_i)\right) & = \sum_if(x_i)_1\otimes f(y_i)_1\otimes f(x_i)_2f(y_i)_2 \\
 & = \sum_if\left((x_i)_1\right)\otimes f\left((y_i)_1\right)\otimes f\left((x_i)_2\right)f\left((y_i)_2\right) \\
 & = \sum_if\left((x_i)_1\right)\otimes f\left((y_i)_1\right)\otimes f\left((x_i)_2(y_i)_2\right) \\
 & = \sum_if\left(x_i\right)\otimes f\left(y_i\right)\otimes f\left(1_H\right) \\
 & = \sum_if\left(x_i\right)\otimes f\left(y_i\right)\otimes 1_L.
\end{align*}
Then, a direct computation shows that $\epsilon _{ {L\hat{\otimes}L}}\circ(co(f)\otimes f) = (f\otimes f) \circ \epsilon _{ {H\hat{\otimes}H}}$. Indeed,
\begin{align*}
\sum_i & \left(\epsilon _{ {L\hat{\otimes}L}}\circ(co(f)\otimes f)\right) ((x_i\otimes y_i)\otimes z_i) = \sum_i\epsilon _{ {L\hat{\otimes}L}}\left((f(x_i)\otimes f(y_i))\otimes f(z_i)\right) \\ 
 & = \sum_i\omega_L^{-1}\left(f(x_i)_1\otimes f(y_i)_1\otimes f(z_i)_1\right)(f(x_i)_2\otimes f(y_i)_2f(z_i)_2) \\
 & = \sum_i\omega_L^{-1}\left(f\left((x_i)_1\right)\otimes f\left((y_i)_1\right)\otimes f\left((z_i)_1\right)\right)\left(f\left((x_i)_2\right)\otimes f\left((y_i)_2(z_i)_2\right)\right) \\
 & \stackrel{(*)}{=} \sum_i\omega_H^{-1}\left((x_i)_1\otimes (y_i)_1\otimes (z_i)_1\right)\left(f\left((x_i)_2\right)\otimes f\left((y_i)_2(z_i)_2\right)\right) \\ 
 & = \sum_i(f\otimes f)\left(\omega_H^{-1}\left((x_i)_1\otimes (y_i)_1\otimes (z_i)_1\right)\left((x_i)_2\otimes (y_i)_2(z_i)_2\right)\right) \\
 & = \sum_i\left((f\otimes f)\circ \epsilon _{ {H\hat{\otimes}H}}\right)\left((x_i\otimes y_i)\otimes z_i\right)
\end{align*}
Therefore, by applying $\varepsilon_L\otimes L\otimes \varepsilon_L$ to both sides of the following relation
$$
\sum f\left(h_{1}\right)\otimes f\left(S_H\left( h_{2}\right)\right) \otimes f\left(h_{3}\right) = \sum f(h)_{1}\otimes S_L\left( f(h)_{2}\right) \otimes f(h)_{3},
$$
we get that $f\circ S_H = S_L \circ f$, proving the second claim. Now, assume that $S$ and $T$ are two preantipodes for $H$. In light of the second claim applied to the case $(H,S), (H,T)$ and the identity morphism, we conclude that $S =T $, proving the first claim as well. 
\end{invisible}
\end{proof}

\begin{remark}
Notice that the functor $\coinv{(-)}{H}:{^H_{}\M^H_H}\to{^H\M}$ needs not to be well-defined if the functor $H\otimes-$ does not preserve, at least, coreflexive equalizers (\ie equalizers of parallel arrows admitting a common retraction).
\end{remark}

\section{Coquasi-bialgebras with preantipode and rigid monoidal categories}\label{sec:tannaka}

It is well-known that every rigid monoidal category together with a monoidal functor to the category of finitely-generated and projective $\K$-modules gives rise to a Hopf algebra. Via a variant of the same Tannaka-Kre{\u\i}n reconstruction process, it has been shown by Majid in \cite{Majid-reconstruction} that every monoidal category $\cC$ together with a quasi-monoidal functor $\forget: \cC \rightarrow \M_f$ gives rise to a coquasi-bialgebra $H$ instead. A very natural question then is what happens if the category $\cC$ is also rigid.

Our aim in this section is to show how this rigidity is related with the existence of a preantipode on $H$. We will do this without any additional assumption on $\cC$. In particular, we will implicitly allow $\forget(X^\star)$ and $\forget(X)^*$ to be non-isomorphic objects, as it happens for example in \cite[Example 4.5.1]{Schauenburg-HopExt}.

\begin{remark}
The following observation about a previous version of the present paper has been brought to our attention and deserves to be highlighted. Assume that $\K$ is a field and consider a category $\cC$ together with a functor $\forget:\cC\to\M_f$. Let $C$ be the reconstructed coalgebra and denote by ${\forget^C}:\cC\to {^C\M_f}$ the induced functor. Then every finite-dimensional $C$-comodule can be recovered from comodules of the form ${\forget^C(X)}$ by taking finite direct sums, kernels and cokernels (see \cite[Corollary 2.2.9]{Sch-Tannaka}).
\begin{invisible}
What follows is my own proof, mimicking Schauenburg's argument, that the claim is true.

Recall that $C$ can be realized as the quotient
$$
\frac{\bigoplus_{X\in\cC}\forget(X)\otimes \forget(X)^*}{\mathsf{Span}_{\K}\big\{x\otimes \left(\psi\circ \forget(f)\right) - \forget(f)(x)\otimes \psi\mid x\in\forget(X),\psi\in\forget(Y)^*,f\in\cC(X,Y)\big\}}.
$$
Set $\cI_{\cC}$ for the vector subspace we are quotienting by. Let $(M,\rho_M)$ be a finite-dimensional comodule over $C$ with basis $\{e_i\mid i=1,\ldots,n\}$. Consider the (finite) set of coefficients $c_{i,j}$ in $C$ determined by $\rho_{M}(e_i)=\sum_{j=1}^n c_{i,j}\otimes e_j$. Then
$$
c_{i,j}=\sum_{k}^{t_{i,j}}x^k_{i,j}\otimes \varphi^k_{i,j}+\cI_{\cC}, \qquad \varphi^k_{i,j}\in\forget(X_{i,j}^k)^*,x^k_{i,j}\in\forget(X_{i,j}^k).
$$
Let $\cE$ be the full subcategory of $\cC$ such that
$$
\cE_0=\left\{X_{i,j}^k\mid i,j=1,\ldots,n,k=1,\ldots,t_{i,j}\right\}
$$
and let
$$E=\coend(\forget|_\cE)=\frac{\bigoplus_{Y\in\cE}\forget(Y)\otimes \forget(Y)^*}{\cI_{\cE}}.$$
It is a finite-dimensional subcoalgebra of $C$ and $M$ is a finite-dimensional $E$-comodule via the corestriction $\rho_M^E:M\to E\otimes M$ of $\rho_M$. Then $M$, as a $C$-comodule, can be recovered as the equalizer of the diagram in ${^C\M_f}$
$$
\xymatrix @C=40pt{
E\otimes M \ar@<+0.5ex>[r]^-{E\otimes \rho_M} \ar@<-0.5ex>[r]_-{\Delta_E\otimes M} & E\otimes E\otimes M,
}
$$
where the comodule structures come from the left-most tensor factors. Now, for every $X,Y\in\cE$ we can choose morphisms $f_{X,Y}^m\in\cC({X},{Y})$ such that the $\forget(f_{X,Y}^m)$'s form a basis for the vector subspace $\langle\forget\left(\cC({X},{Y})\right)\rangle$ generated by $\forget\left(\cC({X},{Y})\right)$ in $\Homk(\forget(X),\forget(Y))$, as the following lemma shows.

\begin{lemma}
Given $X,Y$ in $\cC$, there exists $t\leq\mathsf{dim}(\Homk(\forget(X),\forget(Y)))$ and $t$ arrows $f_1,\ldots,f_t\in \cC(X,Y)$ such that $\{\forget(f_1),\ldots,\forget(f_t)\}$ is a basis for $\langle\forget\left(\cC({X},{Y})\right)\rangle$.
\end{lemma}

\begin{proof}
Let $\cP:=\left\{S\subseteq \forget(\cC(X,Y))\mid S \text{ is free in }\Homk(\forget(X),\forget(Y))\right\}$. This is a poset under ordinary inclusion. Of course, every chain eventually stabilizes because $\langle\forget\left(\Hom{}{}{\cC}{}{X}{Y}\right)\rangle$ is finite-dimensional, whence it admits an upper bound. Therefore, in view of Zorn's Lemma, $\cP$ admits a maximal element $S_m:=\{\forget(f_1),\ldots,\forget(f_t)\}$ and obviously $\langle \forget(f_1),\ldots,\forget(f_t) \rangle\subseteq \langle\forget\left(\cC({X},{Y})\right)\rangle$. Conversely, if $\forget(g)\notin \langle \forget(f_1),\ldots,\forget(f_t) \rangle$ then $\{\forget(f_1),\ldots,\forget(f_t),\forget(g)\}\in\cP$ and strictly contains $S_m$, which is a contradiction. Thus, $\forget\left(\cC({X},{Y})\right)\subseteq \langle \forget(f_1),\ldots,\forget(f_t) \rangle$ and hence $\langle \forget(f_1),\ldots,\forget(f_t) \rangle = \langle\forget\left(\cC({X},{Y})\right)\rangle$.
\end{proof}

\noindent Therefore, since $\cI_{\cE}$ can be shown to coincide with
$$
\mathsf{Span}_{\K}\left\{\forget(f_{X,Y}^m)(x)\otimes \varphi - x\otimes \forget(f_{X,Y}^m)^*(\varphi)\mid x\in\forget(X),\varphi\in\forget(Y)^*,f_{X,Y}^m\in\cC(X,Y)\right\},
$$
$E$ can be recovered as the coequalizer in ${\fvect}$ of the pair of arrows
\begin{equation}\label{eq:horrible}
\xymatrix @C=30pt{
{\displaystyle \bigoplus_{f_{X,Y}^m \, (X,Y\in\cE)}\forget \left( \mathrm{dom}\left( f_{X,Y}^h\right)  \right)\otimes \forget \left( \mathrm{cod}\left( f_{X,Y}^h\right) \right) ^{\ast }}  \ar@<+0.5ex>[r]^-{F_{Q}}\ar@<-0.5ex>[r]_-{F_{Z}} & {\displaystyle\bigoplus_{X\in \cE_{0}}\forget \left( X\right)\otimes \forget \left(X\right) ^{\ast }   }
}
\end{equation}
where $F_{Q}$ is the unique map satisfying
$$
F_{Q}\circ \mathrm{inj}_{f}=\mathrm{inj}_{\mathrm{dom}\left( f\right) }\circ\left( \forget \left( \mathrm{dom}\left( f\right) \right)\otimes\forget \left( f\right) ^{\ast } \right) 
$$
and $F_{Z}$ the unique such that
$$
F_{Z}\circ \mathrm{inj}_{f}=\mathrm{inj}_{\mathrm{cod}\left( f\right) }\circ\left( \forget \left( f\right) \otimes \forget \left( \mathrm{cod}\left( f\right) \right)^{\ast } \right) .
$$
\begin{personal}
The idea is that since $\circ$ for vector spaces is bilinear, it is enough to know the relations in $\cI_{\cE}$ for a family of $f$ such that all the others are linear combinations of these.
\end{personal}
Every $\forget \left( X\right)\otimes \forget \left(Y\right) ^{\ast }$ is a left $C$-comodule with structure coming from the left-most tensor factor, the coproduct of comodules is again a comodule (\cite[\S3.7]{BrWi}) and all the arrows involved in the realization of $E$ are $C$-colinear, whence it can be recovered as coequalizer in ${^C\M_f}$.

Taking advantage of the fact that all the $\forget(Y)^*$ are finite-dimensional, we may consider the isomorphism of $C$-comodules $\forget(X)\otimes \forget(Y)^*\cong \forget(X)^{\dim(\forget(Y)^*)}$ and hence claim that $E$ is a cokernel of a morphism of comodules that are finite biproducts of comodules of the form ${\forget^C}(X)$ for $X\in \cC$. In the same way, by using $E\otimes M\cong E^{\dim(M)}$ and $E\otimes E\otimes M\cong E^{\dim(E)\dim(M)}$ we may claim that $M$ is the kernel of a morphism between comodules that can be constructed by taking finite biproducts and cokernel of comodules of the form ${\forget^C}(X)$ for $X\in \cC$. Summing up, every finite-dimensional $C$-comodule can be recovered by taking finite biproducts, kernels and cokernels of comodules coming from objects in $\cC$ (see also \cite[Corollary 2.2.9]{Sch-Tannaka}).
\end{invisible}
Since in an abelian monoidal category with exact tensor product (\eg the category of comodules over a $\K$-coalgebra), the family of rigid objects is closed under finite biproducts, kernels and cokernels, if $\cC$ is rigid then ${^C\M_f}$ has to be rigid as well and hence $C$ admits a preantipode in light of \cite[Theorem 2.6]{Sch-TwoChar}.

\begin{invisible}
Let us prove briefly the claim about abelian monoidal categories, before continuing. Let $\mathcal{C}$ be an abelian monoidal category with exact tensor product and let $f:X\rightarrow Y$ be a morphism between rigid objects in $\mathcal{C}$. We want to prove that $\ker \left( f\right) ^{\star }=\coker\left( f^{\star }\right) $ and that $\coker\left( f\right) ^{\star}=\ker \left( f^{\star }\right) $. Let us adopt the following notation
\begin{gather*}
0 \longrightarrow \ker \left( f\right) \overset{k}{\longrightarrow }X\overset{f}{\longrightarrow }Y\overset{c}{\longrightarrow }\coker\left( f\right) \longrightarrow 0, \\
0 \longrightarrow \ker \left( f^{\star }\right) \overset{k_{\star }}{\longrightarrow }Y^{\star }\overset{f^{\star }}{\longrightarrow }X^{\star }\overset{c_{\star }}{\longrightarrow }\coker\left( f^{\star }\right)\longrightarrow 0.
\end{gather*}
First of all, we define $\db{k}:\I \rightarrow \coker\left( f\right) ^{\star }\otimes \ker \left( f\right) $ as follows. Consider the composition
\begin{equation*}
\I \overset{\db{X}}{\longrightarrow }X^{\star }\otimes X\overset{c_{\star }\otimes X}{\longrightarrow }\coker\left( f^{\star}\right) \otimes X.
\end{equation*}
We have that
\begin{align*}
\left( \coker\left( f^{\star }\right) \otimes f\right) \circ \left(c_{\star }\otimes X\right) \circ \db{X} &=\left( c_{\star }\otimes Y\right) \circ \left( X^{\star }\otimes f\right) \circ \db{X} \\
&=\left( c_{\star }\otimes Y\right) \circ \left( f^{\star }\otimes Y\right)\circ \db{Y}=0
\end{align*}
and hence $\left( c_{\star }\otimes X\right) \circ \db{X}$ factors through the kernel of $\coker\left( f^{\star }\right) \otimes f$, i.e. there exists a unique $\db{k}$ such that 
\begin{equation*}
\left( c_{\star }\otimes X\right) \circ \db{X}=\left( \coker\left( f^{\star }\right) \otimes k\right) \circ \db{k}.
\end{equation*}
Secondly, we define $\ev{k}:\ker \left( f\right) \otimes \coker\left( f^{^{\star }}\right) \rightarrow \I $ as follows. Consider the composition
\begin{equation*}
\ker \left( f\right) \otimes X^{\star }\overset{k\otimes X^{\star }}{\longrightarrow }X\otimes X^{\star }\overset{\ev{X}}{\longrightarrow }\I .
\end{equation*}
We have that
\begin{align*}
\ev{X}\circ \left( k\otimes X^{\star }\right) \circ \left( \ker\left( f\right) \otimes f^{\star }\right)  &=\ev{X}\circ \left(X\otimes f^{\star }\right) \circ \left( k\otimes Y^{\star }\right)  \\
&=\ev{Y}\circ \left( f\otimes Y^{\star }\right) \circ \left(k\otimes Y^{\star }\right) =0
\end{align*}
and hence $\ev{X}\circ \left( k\otimes X^{\star }\right) $ factors through the cokernel of $\ker \left( f\right) \otimes f^{\star }$, i.e. there exists a unique $\ev{k}$ such that
\begin{equation*}
\ev{k}\circ \left( \ker \left( f\right) \otimes c_{\star }\right) =\ev{X}\circ \left( k\otimes X^{\star }\right) .
\end{equation*}
Let us check that these satisfy the zigzag identities. We compute
\begin{align*}
k\circ & \left( \ev{k}\otimes \ker \left( f\right) \right) \circ\left( \ker \left( f\right) \otimes \db{k}\right) \\
& =\left( \ev{k}\otimes X\right) \circ \left( \ker \left( f\right) \otimes \coker\left( f^{\star }\right) \otimes k\right) \circ \left( \ker\left( f\right) \otimes \db{k}\right)  \\
&=\left( \ev{k}\otimes X\right) \circ \left( \ker \left( f\right)\otimes c_{\star }\otimes X\right) \circ \left( \ker \left( f\right)\otimes \db{X}\right)  \\
&=\left( \ev{X}\otimes X\right) \circ \left( k\otimes X^{\star}\otimes X\right) \circ \left( \ker \left( f\right) \otimes \db{X}\right)  \\
&=\left( \ev{X}\otimes X\right) \circ \left( X\otimes \db{X}\right) \circ k=k,
\end{align*}
from which we deduce (since $k$ is mono) that $\left( \ev{k}\otimes\ker \left( f\right) \right) \circ \left( \ker \left( f\right) \otimes \db{k}\right) =\mathrm{id}_{\ker \left( f\right) }$, and
\begin{align*}
\left( \coker\left( f^{\star }\right) \otimes \ev{k}\right) & \circ \left( \db{k}\otimes \coker\left( f^{\star }\right)\right) \circ c_{\star } \\
&=\left( \coker\left( f^{\star }\right)\otimes \ev{k}\right) \circ \left( \coker\left( f^{\star}\right) \otimes \ker \left( f\right) \otimes c_{\star }\right) \circ \left( \db{k}\otimes X^{\star }\right)  \\
&=\left( \coker\left( f^{\star }\right) \otimes \ev{X}\right) \circ \left( \coker\left( f^{\star }\right) \otimes k\otimes X^{\star }\right) \circ \left( \db{k}\otimes X^{\star}\right)  \\
&=\left( \coker\left( f^{\star }\right) \otimes \ev{X}\right) \circ \left( c_{\star }\otimes X\otimes X^{\star }\right) \circ \left( \db{X}\otimes X^{\star }\right)  \\
&=c_{\star }\circ \left( X\otimes \ev{X}\right) \circ \left( \db{X}\otimes X^{\star }\right) =c_{\star },
\end{align*}
from which we deduce that $\left( \coker\left( f^{\star }\right) \otimes \ev{k}\right) \circ \left( \db{k}\otimes \coker\left( f^{\star }\right) \right) =\mathrm{id}_{\coker\left(f^{\star }\right) }$. This shows that $\coker\left( f^{\star}\right) $ is a right dual for $\ker \left( f\right) $. Conversely, let us consider
\begin{equation*}
\I \overset{\db{Y}}{\longrightarrow }Y^{\star }\otimes Y\overset{Y^{\star }\otimes c}{\longrightarrow }Y^{\star }\otimes \coker\left( f\right) .
\end{equation*}
Since
\begin{align*}
\left( f^{\star }\otimes \coker\left( f\right) \right) \circ \left(Y^{\star }\otimes c\right) \circ \db{Y} &=\left( X^{\star}\otimes c\right) \circ \left( f^{\star }\otimes Y\right) \circ \db{Y} \\
&=\left( X^{\star }\otimes c\right) \circ \left( X^{\star }\otimes f\right) \circ \db{X}=0
\end{align*}
there exists a unique morphism $\I \overset{\db{c}}{\longrightarrow }\ker \left( f^{\star }\right) \otimes \coker\left(f\right) $ such that
\begin{equation*}
\left( k_{\star }\otimes \coker\left( f\right) \right) \circ \db{c}=\left( Y^{\star }\otimes c\right) \circ \db{Y}.
\end{equation*}
Moreover, if we consider 
\begin{equation*}
Y\otimes \ker \left( f^{\star }\right) \overset{Y\mathsf{\otimes k}_{\star }}	{\longrightarrow }Y\otimes Y^{\star }\overset{\ev{Y}}{\longrightarrow }\I 
\end{equation*}
then
\begin{align*}
\ev{Y}\circ \left( Y\otimes k_{\star }\right) \circ \left( f\otimes \ker \left( f^{\star }\right) \right)  &=\ev{Y}\circ \left( f\otimes Y^{\star }\right) \circ \left( X\otimes k_{\star }\right)  \\
&=\ev{X}\circ \left( X\otimes f^{\star }\right) \circ \left( X\otimes k_{\star }\right) =0,
\end{align*}
whence there exists a unique morphism $\coker\left( f\right) \otimes \ker \left( f^{\star }\right) \overset{\ev{c}}{\longrightarrow } \I $ such that
\begin{equation*}
\ev{c}\circ \left( c\otimes \ker \left( f^{\star }\right) \right) =\ev{Y}\circ \left( Y\otimes k_{\star }\right) .
\end{equation*}
Let us show these satisfy the zigzag identities as well. From
\begin{align*}
\left( \ev{c}\otimes \coker\left( f\right) \right) & \circ\left( \coker\left( f\right) \otimes \db{c}\right) \circ c \\
 & =\left( \ev{c}\otimes \coker\left( f\right) \right) \circ\left( c\otimes \ker \left( f^{\star }\right) \otimes \coker\left(f\right) \right) \circ \left( Y\otimes \db{c}\right)  \\
&=\left( \ev{Y}\otimes \coker\left( f\right) \right) \circ \left( Y\otimes k_{\star }\otimes \coker\left( f\right) \right) \circ \left( Y\otimes \db{c}\right)  \\
&=\left( \ev{Y}\otimes \coker\left( f\right) \right) \circ \left( Y\otimes Y^{\star }\otimes c\right) \circ \left( Y\otimes \db{Y}\right)  \\
&=c\circ \left( \ev{Y}\otimes Y\right) \circ \left( Y\otimes \db{Y}\right) =c
\end{align*}
and
\begin{align*}
k_{\star } &\circ \left( \ker \left( f^{\star }\right) \otimes \ev{c}\right) \circ \left( \db{c}\otimes \ker \left( f^{\star}\right) \right) \\
& =\left( Y^{\star }\otimes \ev{c}\right) \circ \left( k_{\star }\otimes \coker\left( f\right) \otimes \ker \left( f^{\star }\right) \right) \circ \left( \db{c}\otimes \ker \left(f^{\star }\right) \right)  \\
&=\left( Y^{\star }\otimes \ev{c}\right) \circ \left( Y^{\star}\otimes c\otimes \ker \left( f^{\star }\right) \right) \circ \left(\db{Y}\otimes \ker \left( f^{\star }\right) \right)  \\
&=\left( Y^{\star }\otimes \ev{Y}\right) \circ \left( Y^{\star}\otimes Y\otimes k_{\star }\right) \circ \left( \db{Y}\otimes \ker\left( f^{\star }\right) \right)  \\
&=\left( Y^{\star }\otimes \ev{Y}\right) \circ \left( \db{Y}\otimes Y^{\star }\right) \circ k_{\star }=k_{\star }
\end{align*}
we conclude that $\left( \ev{c}\otimes \coker\left(f\right) \right) \circ \left( \coker\left( f\right) \otimes \db{c}\right) =\mathrm{id}_{\coker\left( f\right) }$ and $\left(\ker \left( f^{\star }\right) \otimes \ev{c}\right) \circ \left( \db{c}\otimes \ker \left( f^{\star }\right) \right) =\id_{\ker \left( f^{\star }\right) }$, whence $\ker \left( f^{\star }\right) $ is a right dual object for $\coker\left( f\right) $.
\end{invisible}

This purely categorical argument shows that, at least when $\K$ is a field, the validity of our reconstruction theorem should not be surprising. However, the main focus in the present paper is not only on proving the existence of a preantipode for the reconstructed coquasi-bialgebra (even when $\K$ is just a commutative ring), but also to provide an explicit construction of it (thing that, up to our knowledge, cannot be obtained from the foregoing approach).
\end{remark}

\subsection{The classical reconstruction}\label{ssec:classical}

The results in this subsection are well-known. Nevertheless, we retrieve the main steps of the classical reconstruction process for the sake of the unaccustomed reader. We refer to \cite{Majid-reconstruction} and \cite{Sch-Tannaka} for further details.

Let $\left( \cC ,\boxtimes ,\I ,\calpha ,\clambda,\crho \right) $ be an essentially small monoidal category equipped with a quasi-monoidal functor $\forget:\cC \rightarrow \fvect $ from $\cC $ into the category of finitely-generated and projective $\K$-modules. This means that in $ \M $ we have a family of isomorphisms $\varphi _{X,Y}:\forget\left( X\right) \otimes \forget\left( Y\right) \rightarrow \forget\left( X\boxtimes Y\right) $, which is natural in both components, and an isomorphism $\varphi_0:\Bbbk \rightarrow \forget\left( \I \right) $ compatible with the left and right unit constraints as in \eqref{eq:unitfunctor}.
For every $\K$-module $V$ and every  $n\geq 1$, denote by $\forget^{n}:\cC ^{n}\rightarrow {\fvect }$ the functor mapping every $n$-uple of objects $\left( X_{1},\ldots, X_{n}\right) $ in $\cC^{n}$ to the tensor product $\forget\left( X_{1}\right) \otimes \cdots \otimes \forget\left( X_{n}\right) $ in $\M$ and by $\Nat \left( \forget^{n},V\otimes \forget^{n}\right) $ the set  of natural transformations between $\forget^{n}$ and the functor $V\otimes \forget^{n}:\cC^{n} \rightarrow  \M $, sending $\left( X_{1},\ldots, X_{n}\right) $ to $V\otimes \forget^{n}( X_{1},\ldots, X_{n})$. It turns out the functor $\Nat \left( \forget^{n},-\otimes \forget^{n}\right) : \M \rightarrow \Set $ is represented by the $n$-fold tensor product $H_{\forget}^{\otimes n}$ of a suitable coquasi-bialgebra $H_{\forget}$ via a natural isomorphism 
\begin{equation}\label{eq:repr}
\vartheta^{n}:\Homk\left( H_{\forget}^{\otimes n},-\right) \cong \mathrm{Nat
}\left( \forget^{ n},-\otimes \forget^{ n}\right).
\end{equation}
For all $X_1,\ldots,X_n$ in $\cC$, $V$ in $\M$ and $f\in \Homk\left( H_{\forget}^{\otimes n},V\right)$, this is given explicitly by
\begin{equation*}
{\vartheta _{{ V}}^{ n}\left( f\right) _{{X_{1},\ldots,X_{n}}}  =  \left(  f  \otimes \forget^{ n}\left(X_{1},\ldots, X_{n}\right) \right)  \mb{\tau}_{n}  \left( \frho _{{X_{1}}}  \otimes \cdots \otimes  \frho _{{X_{n}}}\right)}
\end{equation*}
where $\frho :=\vartheta _{ {H}}\left( \mathrm{id}_{H}\right) : \forget\rightarrow H\otimes \forget$, $\mb{\tau}_{n}:=\tau _{{\forget^{ n-1}(X_{1},\ldots , X_{n-1}),H}} \circ \cdots \circ \tau_{{\forget( X_{1}),H}}$ and $\tau _{V,W}:V\otimes W\rightarrow W\otimes V$ denotes the natural transformation acting as $\tau_{V,W}(v\otimes w)=w\otimes v$ for every pair of objects $V,W$ in $ \M $.
Since $\forget$ is fixed, we may write $H$ instead of $H_{\forget}$ and we refer to it as the \emph{coendomorphism coquasi-bialgebra} of $\forget$. As a $\K$-module, it is defined to be the coend\footnote{see e.g. \cite[\S IX.6]{Mac} for details about the coend construction} of the functor $\forget \otimes \forget^{\ast }$ from $\cC \times \cC ^{\mathrm{op}}$ to $ \M $. The comultiplication $\Delta$ and the counit $\varepsilon$ are the unique linear maps such that $\vartheta _{{H\otimes H}}\left( \Delta \right)=\left( H\otimes \frho\right) \frho$ and $\vartheta _{\Bbbk }\left( \varepsilon \right)=\mathrm{id}_{\forget}$. The multiplication $m:H\otimes H\rightarrow H$ is uniquely given by the relation $\left( H\otimes \varphi _{{X,Y}}\right)  \vartheta _{{H}}^{2}\left(m\right) _{{X,Y}}=\frho _{{X\boxtimes Y}} \varphi _{{X,Y}}$ while the reassociator $\omega \in \left( H\otimes H\otimes H\right)^{\ast }$ satisfies
\begin{equation}
\varphi _{{X\boxtimes Y,Z}} \left( \varphi _{{X,Y}}\otimes \forget
\left( Z\right) \right)  \vartheta _{\Bbbk }^{3}\left( \omega \right)
_{{X,Y,Z}}= \forget\left( \calpha _{{X,Y,Z}}^{-1}\right)  \varphi _{{X,Y\boxtimes
Z}} \left( \forget\left( X\right) \otimes \varphi _{{Y,Z}}\right)
\label{eq:associativity}
\end{equation}
 for all $X,Y,Z$ in $\cC $. The unit is the unique morphism $u:\Bbbk\rightarrow H$ such that
\begin{equation} \label{eq:defu}
\left(H\otimes \varphi_0\right)  \left(u\otimes \K\right)=\frho_{{\I}} \varphi_0.
\end{equation}
Observe that every $\forget(X)$ is an $H$-comodule via $\rho_{\forget(X)}=\delta_X$ and that $\vartheta_H^2(m)_{X,Y}=\rho_{\forget(X)\otimes\forget(Y)}$ and $u\otimes\K =\rho_{\K}$ are exactly the coactions that makes of ${^H\M}$ a monoidal category. Summing up, we have the following central result.

\begin{theorem}[{\cite[Theorem 2.2]{Majid-reconstruction}}]\label{th:majid}
Let $\left( \cC ,\boxtimes ,\I ,\calpha ,\clambda ,\crho \right) $ be an essentially small monoidal category and let $\left( \forget,\varphi,\varphi_0\right)$ $\forget:\cC \rightarrow \fvect $, be a quasi-monoidal functor. Then there is a coquasi-bialgebra $H$, unique up to isomorphism, universal with the property that $\forget$ factorizes as a monoidal functor $\forget^{H}:\cC \rightarrow {^{H}  \M }$ followed by the forgetful functor. Universal means that if $ H^{\prime }$ is another such coquasi-bialgebra then there is a unique map of coquasi-bialgebras $\epsilon:H\rightarrow H^{\prime }$ inducing a functor ${^{\epsilon}  \M }:{^{H}  \M }\rightarrow {^{H^{\prime }} \M }$ such that ${^{\epsilon} \M }\, \forget^{H} = \forget^{H^{\prime}}:\cC \rightarrow {^{H^{\prime}}  \M }$.
\end{theorem}

In \cite{Majid-reconstruction} there's no explicit reference to the unitality of the multiplication or of the reassociator. Nevertheless, it can be checked that the above constructed maps satisfy all the conditions defining a coquasi-bialgebra.

\begin{remark}
\label{rem:reconstructioncoquasi}
Assume that $\K$ is a field, $\cC $ is already the category ${^{B}\fvect }$ of finite-dimensional comodules over a coquasi-bialgebra $B$ and $\forget$ is already the forgetful functor $\cU:{^{B} \M }_{f}\rightarrow {\fvect }$. Then $B$ itself is a representing object for $\Nat \left( \cU,-\otimes \cU\right)$ (cf. e.g. \cite[Lemma 2.2.1]{Sch-Tannaka}). In this case, the (co)multiplication, the (co)unit and the reassociator of $B$ already satisfy the defining relations for $\Delta$, $\varepsilon$, $m$, $u$ and $\omega$ stated in \S\ref{ssec:classical}, whence they are the unique coquasi-bialgebra structure maps induced on the $\K$-vector space $B$ by the isomorphisms \eqref{eq:repr} in view of Theorem \ref{th:majid}. 
\end{remark}

\begin{invisible}
\begin{lemma}
The reassociator $\omega$ and the multiplication $m$ are unital.
\end{lemma}

\begin{proof}
Since the detailed computations would be too technical and out of the purposes of this paper, let us give only a sketch of the proof. 
Denote temporarily by $H^{(i)}$ the $i$-th tensorand in $H^{\otimes n}$ for some $n\geq 1$ and some $1\leq i\leq n$ and by $\varphi_0^{(i)}$ the morphism $\varphi_0$ applied in the $i$-th position of a tensor product. The following computation
\begin{equation*}
\gbeg{11}{12}
\gvac{1}\got{1}{\micro{X_1}}\gvac{1}\got{1}{\micro{X_2}}\got{3}{\micro{\cdots_{}}}\gvac{1}\got{2}{\micro{\cdots_{}}}\got{1}{\micro{X_n}}\gnl
\gvac{1}\gcl{1}\gvac{1}\gcl{1}\gvac{3}\gmpub{\varphi_0}\gvac{2}\gcl{1}\gnl
\glcm\glcm\gvac{2}\glcm\gvac{1}\glcm\gnl
\gcl{1}\gcl{1}\gcl{1}\gcl{1}\gcn{2}{1}{5}{1}\gvac{1}\gcl{1}\gcn{2}{1}{3}{1}\gcl{1}\gnl
\gcl{1}\gbr\gbr\gvac{2}\gbr\gvac{1}\gcl{1}\gnl
\gcl{1}\gcl{1}\gbr\gcn{1}{1}{1}{3}\gvac{1}\gcn{1}{1}{3}{1}\gvac{1}\gcl{1}\gvac{1}\gcl{1}\gnl
\gcl{1}\gcl{1}\gcl{1}\gcn{1}{1}{1}{3}\gvac{1}\gbr\gvac{1}\gcl{1}\gvac{1}\gcl{1}\gnl
\gcl{1}\gcl{1}\gcl{1}\gvac{1}\gbr\gcl{1}\gvac{1}\gcl{1}\gvac{1}\gcl{1}\gnl
\gcl{1}\gcl{1}\gcl{1}\gvac{1}\gcl{1}\gcl{1}\gcl{1}\gvac{1}\gcl{1}\gvac{1}\gcl{1}\gnl
\gsbox{5}\gnot{\hspace{16mm}f}\gvac{5}\gcl{1}\gcl{1}\gvac{1}\gcl{1}\gvac{1}\gcl{1}\gnl
\gvac{2}\gcl{1}\gvac{2}\gcl{1}\gcl{1}\gvac{1}\gcl{1}\gvac{1}\gcl{1}\gnl
\gvac{2}\got{1}{\micro{V_{}}}\gvac{2}\got{1}{\micro{X_1}}\got{1}{\micro{X_2}}\got{1}{\micro{\cdots_{}}}\got{1}{\micro{\I_{}}}\got{1}{\micro{\cdots_{}}}\got{1}{\micro{X_n}}
\gend
=
\gbeg{11}{11}
\gvac{1}\got{1}{\micro{X_1}}\gvac{1}\got{1}{\micro{X_2}}\got{3}{\micro{\cdots_{}}}\gvac{1}\got{2}{\micro{\cdots_{}}}\got{1}{\micro{X_n}}\gnl
\glcm\glcm\gvac{2}\gu{1}\gmpub{\varphi_0}\gvac{1}\glcm\gnl
\gcl{1}\gcl{1}\gcl{1}\gcl{1}\gcn{2}{1}{5}{1}\gvac{1}\gcl{1}\gcn{2}{1}{3}{1}\gcl{1}\gnl
\gcl{1}\gbr\gbr\gvac{2}\gbr\gvac{1}\gcl{1}\gnl
\gcl{1}\gcl{1}\gbr\gcn{1}{1}{1}{3}\gvac{1}\gcn{1}{1}{3}{1}\gvac{1}\gcl{1}\gvac{1}\gcl{1}\gnl
\gcl{1}\gcl{1}\gcl{1}\gcn{1}{1}{1}{3}\gvac{1}\gbr\gvac{1}\gcl{1}\gvac{1}\gcl{1}\gnl
\gcl{1}\gcl{1}\gcl{1}\gvac{1}\gbr\gcl{1}\gvac{1}\gcl{1}\gvac{1}\gcl{1}\gnl
\gcl{1}\gcl{1}\gcl{1}\gvac{1}\gcl{1}\gcl{1}\gcl{1}\gvac{1}\gcl{1}\gvac{1}\gcl{1}\gnl
\gsbox{5}\gnot{\hspace{16mm}f}\gvac{5}\gcl{1}\gcl{1}\gvac{1}\gcl{1}\gvac{1}\gcl{1}\gnl
\gvac{2}\gcl{1}\gvac{2}\gcl{1}\gcl{1}\gvac{1}\gcl{1}\gvac{1}\gcl{1}\gnl
\gvac{2}\got{1}{\micro{V_{}}}\gvac{2}\got{1}{\micro{X_1}}\got{1}{\micro{X_2}}\got{1}{\micro{\cdots_{}}}\got{1}{\micro{\I_{}}}\got{1}{\micro{\cdots_{}}}\got{1}{\micro{X_n}}
\gend
=
\gbeg{9}{8}
\gvac{1}\got{1}{\micro{X_1}}\gvac{1}\got{1}{\micro{X_2}}\got{2}{\micro{\cdots_{}}}\got{2}{\micro{\cdots_{}}}\got{1}{\micro{X_n}}\gnl
\glcm\glcm\gvac{3}\glcm\gnl
\gcl{1}\gbr\gcn{2}{1}{1}{5}\gvac{1}\gcn{1}{1}{3}{1}\gvac{1}\gcl{1}\gnl
\gcl{1}\gcl{1}\gcn{2}{1}{1}{5}\gvac{1}\gbr\gvac{1}\gcl{1}\gnl
\gcl{1}\gcl{1}\gu{1}\gvac{1}\gbr\gcl{1}\gvac{1}\gcl{1}\gnl
\gsbox{5}\gnot{\hspace{16mm}f}\gvac{5}\gcl{1}\gcl{1}\gvac{1}\gcl{1}\gnl
\gvac{2}\gcl{1}\gvac{2}\gcl{1}\gcl{1}\gmpub{\varphi_0}\gcl{1}\gnl
\gvac{2}\got{1}{\micro{V_{}}}\gvac{2}\got{1}{\micro{X_1}}\got{1}{\micro{X_2}}\got{1}{\micro{\cdots_{}\I_{}\cdots_{}}}\got{1}{\micro{X_n}}
\gend
\end{equation*}
shows that if $f:H^{\otimes n+1}\to V$ is an arrow in $\M$ for $V$ a $\K$-module, then
\begin{equation}\label{eq:thetaunittech}
\vartheta_{ {V}}^{{n+1}}\left(f\right)_{ {X_1,\ldots,\I, X_i ,\ldots,X_n}} \circ \varphi_{0}^{(i)} = \varphi_{0}^{(i+1)} \circ \vartheta_{ {V}}^{{n}}\left(f\circ u_{H^{(i)}}\right)_{ {X_1,\ldots,X_n}}.
\end{equation}
By omitting the constraints $l$ and $r$, the relations in \eqref{eq:unitfunctor} become
\begin{equation*}
\forget\left(\clambda_{ {X}}\right) \circ \varphi_{ {\I ,X}} \circ \left(\varphi_{0}\otimes \forget(X)\right) = \id_{\forget(X)} =  \forget\left(\crho_{ {X}}\right) \circ \varphi_{ {X ,\I}} \circ \left(\forget(X)\otimes \varphi_{0}\right),
\end{equation*}
so that Equation \eqref{eq:thetaunittech} can be rewritten as
\begin{subequations}\label{eq:thetan}
\begin{gather}
\forget\left(\clambda_{ {X_{i}}}\right)\circ \varphi_{ {\I ,X_{i}}}\circ \vartheta_{ {V}}^{{n+1}}\left(f\right)_{ {X_1,\ldots,\I ,X_i,\ldots,X_n}} = \vartheta_{ {V}}^{{n}}\left(f\circ u_{H^{(i)}}\right)_{ {X_1,\ldots,X_n}}\circ \forget\left(\clambda_{ {X_i}}\right)\circ \varphi_{ {\I ,X_i}}\quad\text{or} \label{eq:thetan1} \\
\forget\left(\crho_{ {X_{i-1}}}\right)\circ \varphi_{ {X_{i-1},\I }}\circ \vartheta_{ {V}}^{{n+1}}\left(f\right)_{ {X_1,\ldots,X_{i-1},\I ,\ldots,X_n}} = \vartheta_{ {V}}^{{n}}\left(f\circ u_{H^{(i)}}\right)_{ {X_1,\ldots,X_n}}\circ \forget\left(\clambda_{ {X_{i}}}\right)\circ \varphi_{ {\I,X_{i}}} \label{eq:thetan2}
\end{gather}
\end{subequations}
(notice that Equation \eqref{eq:thetan2} holds only if $1<i<n$, but as we will see this is the only case we are interested in). By naturality of $\vartheta_{ {H}}(\id_H)$, one checks that
\begin{align*}
\vartheta _{ {H}} \left( m \circ \left( u\otimes H\right)  \right)_{ {X}} \circ \forget(\clambda_{ {X}}) & \stackrel{\eqref{eq:thetan1}}{=} \left(H\otimes \forget(\clambda_{ {X}})\right) \circ \left(H\otimes \varphi_{ {\I ,X}}\right)\circ \vartheta _{ {H}}^{2}\left(m\right) _{ {\I ,X}}\circ \varphi_{ {\I ,X}}^{-1} \\
 & = \left(H\otimes \forget(\clambda_{ {X}})\right)\circ \vartheta_{ {H}}\left(\id_H\right)_{ {\I \boxtimes X}}=\vartheta_{ {H}}(\id_H)_{ {X}}\circ \forget(\clambda_{ {X}})
\end{align*}
 for all $X$ in $\cM$, which means that $m \circ \left( u\otimes H\right) =\id_{H}$ as desired. Unitality on the other side may be checked similarly. Let us conclude with the unitality of $\omega$. As above, we may compute
\begin{align*}
& \vartheta _{{\Bbbk} }^{2}\left( \omega \circ \left( H\otimes u\otimes H\right) \right)_{X,Y} \\
 & \stackrel{\eqref{eq:thetan2}}{=} \left(\forget(\crho_{ {X}})\otimes \forget(Y)\right) \circ \left(\varphi_{ {X,\I} }\otimes \forget(Y)\right) \circ \vartheta_{ {\Bbbk}}^3(\omega)_{ {X,\I ,Y}} \circ \left(\forget(X)\otimes \varphi_{ {\I ,Y}}^{-1}\right) \circ \left(\forget(X)\otimes \forget\left(\clambda_{ {Y}}^{-1}\right)\right) \\
 &  = \left(\forget(\crho_{ {X}})\otimes \forget(Y)\right) \circ \varphi _{ {X\boxtimes \I ,Y}}^{-1}  \circ \forget\left( \calpha _{ {X,\I ,Y}}^{-1}\right) \circ \varphi _{ {X,\I \boxtimes Y}} \circ \left(\forget(X)\otimes \forget\left(\clambda_{ {Y}}^{-1}\right)\right) \\
 & = \varphi _{ {X,Y}}^{-1}\circ  \forget(\crho_{ {X}}\boxtimes Y) \circ  \forget\left( \calpha _{ {X,
\I ,Y}}^{-1}\right) \circ  \forget\left(X\boxtimes \clambda_{ {Y}}^{-1}\right)\circ  \varphi _{ {X, Y}}\\
 &  = \id_{\forget(X)\otimes \forget(Y)}=\vartheta _{\Bbbk }^{2}\left( m_{\Bbbk }\circ  \left( \varepsilon \otimes
\varepsilon \right) \right)
\end{align*}
 for all $X,Y$ in $\cM$. Thus, $\omega \left( x\otimes 1_{H}\otimes y\right) =\varepsilon \left(
x\right) \varepsilon \left( y\right) $ for all $x,y\in H$.
\end{proof}
\end{invisible}

\subsection{The rigid case}

We recall briefly some facts about rigid objects in a monoidal category.

\begin{definition}
A \emph{right dual
object} $X^{\star  }$ of $X$ in $\cC $ is a triple $\left(
X^{\star  },\ev{X},\db{X}\right) $ in which $
X^{\star  }$ is an object in $\cC $ and $\ev{X}:X\boxtimes X^{\star  }\rightarrow \I $ and $\db{X}:\I \rightarrow X^{\star
}\boxtimes X$ are morphisms in $\cC $, called \emph{evaluation} and \emph{dual basis} respectively, that satisfy
\begin{gather}
\left( \ev{X}\boxtimes X\right) \, \calpha _{ {X,X^{\star 
},X}}^{-1}\, \left( X\boxtimes \db{X}\right) =\mathrm{id}
_{X},  \label{eq:idX} \\
\left( X^{\star  }\boxtimes \ev{X}\right) \, \calpha
_{ {X^{\star  },X,X^{\star  }}}\, \left( \db{X}\boxtimes
X^{\star  }\right) =\mathrm{id}_{X^{\star  }}.  \label{eq:idXstar}
\end{gather}
An object which admits a right dual object is said to be \emph{right rigid} (or \emph{dualizable}). If every object in $\cC $ is right rigid, then we say that $\cC $ is \emph{right rigid}.
\end{definition}

We will often refer to right dual objects simply as \emph{right duals} or just \emph{duals}.


\begin{remark}
Once chosen a right dual object $X^\star$ for every object $X$ in a right rigid monoidal category $\cC $, we have that
the assignment $\left( -\right) ^{\star  }:\cC ^{\mathrm{op}
}\rightarrow \cC $ defines a functor and $\ev{}:\left(
-\right) \boxtimes \left( -\right) ^{\star  }\rightarrow \I $
and\ $\db{}:\I \rightarrow \left( -\right) ^{\star 
}\boxtimes \left( -\right) $ define dinatural transformations\footnote{More precisely, these should be referred to as \emph{wedges}, since they are dinatural transformations to a constant functor. However, we avoided this in order to spare the proliferation of terminology. For the definition of dinatural transformations and wedges we refer to \cite[\S9.4]{Mac}.}, i.e., for every $X,Y$ and $f:X\rightarrow Y$ in $\cC$ we have $\left(f^\star\boxtimes Y\right)\db{Y}=\left(X^\star\boxtimes f\right)\db{X}$ and $\ev{X}\left(X\boxtimes f^\star\right)=\ev{Y}\left(f\boxtimes Y^\star\right)$.
\end{remark}

From now on, let us assume that $\cC$ is right rigid. If we have a different choice $\left( -\right) ^{\vee }:\mathcal{C}^{\mathrm{op}}\rightarrow \cC $ of right dual objects, then we write $\ev{}^{\left( \star  \right) }$ and $\db{}^{\left( \star  \right) }$ to mean the evaluation and dual basis maps associated with the dual $\left( -\right) ^{\star  }$ and $\ev{} ^{\left( \vee \right) }$ and $\db{}^{\left( \vee \right) }$ to mean those associated with $\left( -\right) ^{\vee }$. We know (see e.g. \cite[\S9.3]{Maj}) that for every $X$ in $\cC $, its right dual is unique up to isomorphism whenever it exists, i.e. we have an isomorphism $\kappa _{ {X}}:X^{\star  }\rightarrow X^{\vee }$ in $\cC $ given by the composition
\begin{equation}\label{eq:kappa}
\kappa _{ {X}}:=\crho _{ {X^{\vee }}} \left( X^{\vee }\boxtimes \ev{X}^{\left( \star  \right) }\right)  \calpha _{ {X^{\vee},X,X^{\star  }}} \left( \db{X}^{\left( \vee \right)}\boxtimes X^{\star  }\right)  \clambda _{ {X^{\star}  }}^{-1}.
\end{equation}

\begin{lemma}\label{lemma:uniquenessdual}
The isomorphism $\kappa _{X}:X^{\star  }\rightarrow X^{\vee }$ is natural in $X$ and
the dinatural transformations $\ev{}^{\left( \star  \right) }$, $
\db{}^{\left( \star  \right) }$, $\ev{}^{\left( \vee \right) }$
and $\db{}^{\left( \vee \right) }$ satisfy
\begin{equation}
\left( \kappa \boxtimes \mathrm{id}\right)  \db{}^{\left(
\star  \right) }=\db{}^{\left( \vee \right) }\text{\qquad and\qquad }
\ev{}^{\left( \vee \right) } \left( \mathrm{id}\boxtimes \kappa
\right) =\ev{}^{\left( \star  \right) }\text{.}  \label{eq:compkappa}
\end{equation}
\end{lemma}

\begin{notation}
In what follows we will retrieve some computations in terms of braided diagrams in the category of $\K$-modules. We will adopt the following notation
\begin{equation*}
\Delta=\gbeg{2}{3} \got{2}{\micro{H}}\gnl \gcmu\gnl \got{1}{\micro{H}}\got{1}{\micro{H}}\gend,
\quad
\varepsilon=\gbeg{1}{3} \got{1}{\micro{H}}\gnl \gcu{1}\gnl \gvac{1} \gend,
\quad
u=\gbeg{1}{3} \gvac{1}\gnl \gu{1}\gnl \got{1}{\micro{H}} \gend,
\quad
m=\gbeg{2}{3} \got{1}{\micro{H}}\got{1}{\micro{H}}\gnl \gmu\gnl \got{2}{\micro{H}} \gend,
\quad
\tau_{\micro{V},\micro{W}}=\gbeg{2}{3} \got{1}{\micro{V}}\got{1}{\micro{W}}\gnl \gbr\gnl \got{1}{\micro{W}}\got{1}{\micro{V}} \gend,
\quad
\frho_\micro{X}=\gbeg{2}{3} \gvac{1}\got{1}{\micro{X}}\gnl \glcm\gnl \got{1}{\micro{H}}\got{1}{\micro{X}} \gend.
\end{equation*}
We will also omit to write the functor $\forget$ in braided diagrams.
\end{notation}

Henceforth and unless stated otherwise, we assume also that a choice $(-)^\star$ of dual objects has been performed. Let us consider the following maps
\begin{equation}
\ev{\forget\left( X\right)}:=\varphi_0^{-1} \forget\left( 
\ev{X}\right)  \varphi _{ {X,X^{\star}  }} \qquad \text{and} \qquad \db{\forget\left( X\right)} :=\varphi _{ {X^{\star 
},X}}^{-1} \forget\left( \db{X}\right)  \varphi_0,
\label{eq:evaluation}
\end{equation}
which we will represent simply as $\ev{\forget(X)}=\gbeg{2}{2} \got{1}{\micro{X}}\got{1}{\micro{X^\star}}\gnl \gev \gend$ and $\db{\forget(X)}=\gbeg{2}{2}  \gdb\gnl \got{1}{\micro{X^\star}}\got{1}{\micro{X}} \gend$.

These do not endow $\forget\left( X^{\star  }\right) $ with a
structure of right dual object of $\forget\left( X\right) $ in the
category $\M$ because the functor $\forget:\cC \rightarrow  \M $ does not satisfy the associativity
condition \eqref{eq:assfunctor}. Nevertheless, we have the following result, whose proof follows easily from the definitions and the dinaturality of  $\ev{}$ and $\db{}$.

\begin{lemma}
\label{Lemma:dinaturality}The assignments $\ev{\forget\left(
X\right) }$ and $\db{\forget\left( X\right)}$ defined in \eqref
{eq:evaluation} give rise to dinatural
transformations $\ev{\forget\left( -\right) }:\forget
\otimes \forget^{\star  }\rightarrow \Bbbk $ and $\db{\forget\left( -\right) }:\Bbbk \rightarrow \forget^{\star 
}\otimes \forget$.
\end{lemma}

\begin{remark}\label{rem:rigidity}
Recall that if $\left( \cF,\phi ,\phi _{0}\right) :\left( 
\cC ,\boxtimes ,\I \right) \rightarrow \left( \mathcal{D}
,\circledast ,\mathbb{J}\right) $ is a monoidal functor between monoidal
categories and if $X$ in $\cC $ has a right dual $\left( X^{\star  },\ev{X},\db{X}\right) $, then $\cF\left( X\right) $ is right rigid with dual object $
\cF\left( X^{\star  }\right) $ and structure maps
\begin{equation*} 
\ev{\cF\left( X\right)}= \phi _{0}^{-1}\, \cF\left( \ev{X}\right) \, \phi
_{ {X,X^{\star}  }}\qquad \text{and}\qquad \db{\cF\left(
X\right) }=\phi _{ {X^{\star  },X}}^{-1}\, \cF\left( \db{X}\right) \, \phi _{0}
\end{equation*}
(cf. e.g. \cite[page 86]{Street}). Therefore, even if $\forget\left(X^{\star  }\right) $ is not a right dual of $\forget\left( X\right) $
in $ \M $, $\left( \forget\left( X^{\star  }\right) ,\frho_{ {X^{\star}  }}\right) $ is a right dual of $\left( \forget\left(X\right) ,\frho _{ {X}}\right) $ in ${^{H} \M }$ because $\forget^{H}:\cC \rightarrow {^{H} \M }$ is monoidal.
Evaluation and coevaluation maps are the same given in \eqref{eq:evaluation}
and they are morphisms of comodules. In particular,
\begin{gather}
\gbeg{6}{7}
\gvac{1}\got{1}{\micro{X}}\gvac{4}\gnl
\gvac{1}\gcl{1}\gvac{1}\gwdb{3}\gnl
\glcm\glcm\glcm\gnl
\gcl{1}\gbr\gbr\gcl{1}\gnl
\gcl{1}\gcl{1}\gbr\gcl{1}\gcl{1}\gnl
\gsbox{3}\gnot{\hspace{9mm}\omega}\gvac{3}\gev\gcl{1}\gnl
\gvac{5}\got{1}{\micro{X}}
\gend
=
\gbeg{1}{3}
\got{1}{\micro{X}}\gnl
\gcl{1}\gnl
\got{1}{\micro{X}}
\gend,
\qquad
\gbeg{6}{7}
\gvac{5}\got{1}{\micro{X^\star}}\gnl
\gvac{1}\gwdb{3}\gvac{1}\gcl{1}\gnl
\glcm\glcm\glcm\gnl
\gcl{1}\gbr\gbr\gcl{1}\gnl
\gcl{1}\gcl{1}\gbr\gcl{1}\gcl{1}\gnl
\gsbox{3}\gnot{\hspace{9mm}\omega^{-1}}\gvac{3}\gcl{1}\gev\gnl
\gvac{3}\got{1}{\micro{X^{\star}}}\gvac{2}
\gend
=
\gbeg{1}{3}
\got{1}{\micro{X^{\star}}}\gnl
\gcl{1}\gnl
\got{1}{\micro{X^{\star}}}
\gend, \label{eq:Xdual} \\
\gbeg{4}{5}
\gvac{1}\gwdb{3}\gnl
\glcm\glcm\gnl
\gcl{1}\gbr\gcl{1}\gnl
\gmu\gcl{1}\gcl{1}\gnl
\got{2}{\micro{H}}\got{1}{\micro{X^\star}}\got{1}{\micro{X}}
\gend
=
\gbeg{3}{2}
\gu{1}\gdb\gnl
\got{1}{\micro{H}}\got{1}{\micro{X^\star}}\got{1}{\micro{X}}
\gend,
\qquad
\gbeg{4}{5}
\gvac{1}\got{1}{\micro{X}}\gvac{1}\got{1}{\micro{X^\star}}\gnl
\glcm\glcm\gnl
\gcl{1}\gbr\gcl{1}\gnl
\gmu\gev\gnl
\got{2}{\micro{H}}\gvac{2}
\gend
=
\gbeg{3}{3}
\gvac{1}\got{1}{\micro{X}}\got{1}{\micro{X^\star}}\gnl
\gu{1}\gev\gnl
\got{1}{\micro{H}}\gvac{2}
\gend, \label{eq:dbcolin}
\end{gather}
where \eqref{eq:Xdual} 
encodes relations \eqref{eq:idX} and \eqref{eq:idXstar}.
\end{remark}



\subsection{The natural transformation \texorpdfstring{$\nabla$}{}}

Consider the distinguished natural transformation $\nabla^{\forget}:\Nat(\forget,-\otimes \forget) \rightarrow \Nat(\forget,-\otimes\forget)$
given by
\begin{equation}\label{eq:Nabla}
\nabla^{\forget}_V(\xi)_X=(V\otimes \ev{\forget\left( X\right)}\otimes \forget(X))\, \tau_{{\forget(X),V}}\,  \xi_{{X^\star}} \, (\forget(X)\otimes \db{\forget(X)})
\end{equation}
for all $V$ in $\M$, $\xi\in \Nat(\forget,V\otimes \forget)$ and $X$ in $\cC$ (when it would be clear from the context where to apply a morphism, we will omit to tensor by the identity maps). Graphically,
\begin{equation}
\nabla^{\forget}_{V}(\xi)_X =
\gbeg{4}{6}
\got{1}{\micro{X}}\gvac{3}\gnl
\gcl{1}\gvac{1}\gdb\gnl
\gcl{1}\gsbox{2}\gnot{\hspace{5mm}\xi_{X^\star}}\gvac{2}\gcl{1}\gnl
\gbr\gcl{1}\gcl{1}\gnl
\gcl{1}\gev\gcl{1}\gnl
\got{1}{\micro{V}}\gvac{2}\got{1}{\micro{X}}
\gend.
\end{equation}

\begin{proposition}\label{rem:nabla}
Let $\cC$ and $\cD$ be essentially small right rigid monoidal categories. Let $(\cV,\psi,\psi_0)$, $\cV:\cD\rightarrow\M_f$, be a quasi-monoidal functor and let $(\cG,\zeta,\zeta_0)$, $\cG:\cC\rightarrow\cD$, be a monoidal one. For all $V\in\M$ and $\xi\in\Nat (\cV,V\otimes \cV)$ we have
\begin{equation}\label{eq:Nablacomp}
\nabla^{ {\cV}}_{ {V}}(\xi)\cG=\nabla^{ {\cV\cG}}_{ {V}}(\xi\cG).
\end{equation}
\end{proposition}

\begin{proof}
Assume that we are given a choice of right duals $(-)^\star$ in $\cC$ and $(-)^\vee$ in $\cD$. Since $\cG$ is monoidal we have a natural isomorphism $\kappa_{ {X}}:\cG(X^\star)\rightarrow\cG(X)^\vee$ as in \eqref{eq:kappa}. Note that the composition $\cV\cG$ is still a quasi-monoidal functor with structure isomorphisms $\phi=(\cV\zeta)\circ\psi (\cG\times \cG)$ and $\phi_0=\cV(\zeta_0) \psi_0$. We will need the following relations, which descend from \eqref{eq:compkappa},
\begin{equation}\label{eq:dbcomp} 
\left(\cV\kappa\otimes \cV\cG\right)\circ\db{}(\cV\cG)=(\db{} \cV) \cG \quad \text{ and } \quad \ev{} (\cV\cG)=\left(\ev{} \cV\right) \cG\circ\left(\cV\cG\otimes \cV\kappa\right).
\end{equation}
That is, for every object $X$ in $\cC$ we have
\begin{gather*}
\left(\cV\left(\kappa_{ {X}}\right)\otimes \cV\cG(X)\right) \, \db{\cV\cG(X)}=\db{\cV(\cG(X))},\qquad \ev{\cV\cG(X)} =\ev{\cV(\cG(X))}\, \left(\cV\cG(X)\otimes \cV\left(\kappa_{ {X}}\right)\right). 
\end{gather*}
As a consequence, for every $\xi\in\Nat(\cV,V\otimes \cV)$ we can compute directly
\begin{align*}
& \nabla^{ {\cV}}_{ {V}}\left(\xi\right)_{ {\cG\left(X\right)}} \stackrel{\eqref{eq:Nabla}}{=} \left(V\otimes \ev{\cV\left(\cG\left(X\right)\right)}\otimes \cV\cG\left(X\right)\right) \, \tau_{ {\cV\cG\left(X\right),V}} \, \xi_{ {\cG\left(X\right)^\vee}} \, \left(\cV\cG\left(X\right)\otimes \db{\cV\left(\cG\left(X\right)\right)}\right) \\
 &\stackrel{\eqref{eq:dbcomp}}{=} \left(V\otimes \ev{\cV\left(\cG\left(X\right)\right)}\otimes \cV\cG\left(X\right)\right) \, \tau_{ {\cV\cG\left(X\right),V}}  \, \xi_{ {\cG\left(X\right)^\vee}}\, \cV\left(\kappa_{ {X}}\right) \, \left(\cV\cG\left(X\right)\otimes \db{\cV\cG\left(X\right)}\right) \\
 & \hspace{2pt}\stackrel{(*)}{=}\hspace{2pt} \left(V\otimes \ev{\cV\left(\cG\left(X\right)\right)}\otimes \cV\cG\left(X\right)\right) \, \tau_{ {\cV\cG\left(X\right),V}}  \, \cV\left(\kappa_{ {X}}\right) \, \xi_{ {\cG\left(X^\star\right)}}\, \left(\cV\cG\left(X\right)\otimes \db{\cV\cG\left(X\right)}\right) \\
 & \stackrel{\eqref{eq:dbcomp}}{=}\left(V\otimes \ev{\cV\cG\left(X\right)}\otimes\cV\cG\left(X\right)\right)\, \tau_{ {\cV\cG\left(X\right),V}} \, \xi_{ {\cG\left(X^\star\right)}} \,  \left(\cV\cG\left(X\right)\otimes \db{\cV\cG\left(X\right)}\right) \stackrel{\eqref{eq:Nabla}}{=} \nabla^{ {\cV\cG}}_{ {V}}\left(\xi \cG\right)_{ {X}}
\end{align*}
where in $(*)$ we used the naturality of $\xi$.
\end{proof}

\begin{corollary}\label{lemma:Sunique}
Let $\cC$ be an essentially small right rigid monoidal category and let $\forget:\cC\rightarrow \M_f$ be a quasi-monoidal functor. The natural transformation $\nabla^{{\forget}}$ does not depend on the choice of the dual objects.
\end{corollary}

\begin{proof}
It is enough to take $\cD=\cC$ and $\cG=\id_{\cC}$ in the proof of Proposition \ref{rem:nabla}.
\end{proof}

\begin{remark}
Mimiking \cite{Sch-Tannaka} we may consider a category $\mathfrak{C}$ whose objects are pairs $(\cC,\cU)$ where $\cC$ is an essentially small right rigid monoidal category and $\cU:\cC\to\M_f$ is a quasi-monoidal functor. Morphisms in $\mathfrak{C}$ between two objects $(\cC,\cU)$ and $(\cD,\cV)$ are given by monoidal functors $\cG:\cC\rightarrow \cD$ such that $\cV\cG=\cU$ as quasi-monoidal functors. It follows from Proposition \ref{rem:nabla} that the transformation $\nabla^{\sim}$ introduced in the foregoing is a natural transformation between the functor $\Nat(\sim,-~\otimes \sim):\mathfrak{C}\rightarrow \mathrm{Funct}(\M,\Set)$ sending $(\cC,\cU)$ to $\Nat(\cU,-\otimes \cU)$ and itself.
\end{remark}

\subsection{Rigidity and the preantipode}
By Yoneda Lemma and the fact that $H$ represents the functor $\Nat(\forget,-\otimes \forget)$, there exists a unique natural transformation in $\Nat(\forget,H\otimes \forget)$ which corresponds to $\nabla^{\forget}$ and it is $\nabla^{\forget}_H(\frho)$. Its component at $X$ is
\begin{equation}\label{eq:nabla}
\nabla^{\forget}_{H}(\frho)_X=(H\otimes \ev{\forget\left( X\right)}\otimes \forget(X))\, \tau_{{\forget(X),H}}\,  \frho_{{X^\star}} \, (\forget(X)\otimes \db{\forget(X)}).
\end{equation}
Moreover, there exists a unique linear endomorphism $S$ of $H$ such that
\begin{equation}\label{eq:preantnabla}
\vartheta_{{H}}(S)_X =
\gbeg{2}{4}
\gvac{1}\got{1}{\micro{X}}\gnl
\glcm\gnl
\gmp{S}\gcl{1}\gnl
\got{1}{\micro{H}}\got{1}{\micro{X}}
\gend
=
\gbeg{4}{6}
\got{1}{\micro{X}}\gvac{3}\gnl
\gcl{1}\gvac{1}\gdb\gnl
\gcl{1}\glcm\gcl{1}\gnl
\gbr\gcl{1}\gcl{1}\gnl
\gcl{1}\gev\gcl{1}\gnl
\got{1}{\micro{H}}\gvac{2}\got{1}{\micro{X}}
\gend
=
 \nabla^{\forget}_{H}(\frho)_X.
\end{equation}
Notice that, by naturality of $\vartheta$ and $\nabla^{\forget}$, for all $g:H\rightarrow V$ in $\M$ we have
\begin{equation}
\vartheta_{{V}}(gS) = \nabla^{\forget}_{V}((g\otimes \forget)\,\frho). \label{eq:thetanabla}
\end{equation}

\begin{proposition}\label{lemma:pre1}
The morphism $S$ is a preantipode for $H$.
\end{proposition}

\begin{proof}
Since $\db{\forget(X)}$ is $H$-colinear, it follows that
\begin{equation*}
\gbeg{4}{10}
\gvac{3}\got{1}{\micro{X}}\gnl
\gvac{2}\glcm\gnl
\gvac{1}\gcn{2}{1}{3}{2}\gcl{1}\gnl
\gvac{1}\gcmu\gcl{1}\gnl
\gvac{1}\gmp{S}\gcl{1}\gcl{1}\gnl
\gcn{2}{1}{3}{2}\gcl{1}\gcl{1}\gnl
\gcmu\gcl{1}\gcl{1}\gnl
\gcl{1}\gbr\gcl{1}\gnl
\gmu\gcl{1}\gcl{1}\gnl
\got{2}{\micro{H}}\got{1}{\micro{H}}\got{1}{\micro{X}}
\gend
\stackrel{\eqref{eq:preantnabla}}{=}
\gbeg{4}{10}
\got{1}{\micro{X}}\gvac{3}\gnl
\gcl{1}\gvac{1}\gdb\gnl
\gcl{1}\glcm\gcl{1}\gnl
\gbr\gcl{1}\gcl{1}\gnl
\gcl{1}\gev\gcl{1}\gnl
\gcn{1}{1}{1}{2}\gvac{1}\glcm\gnl
\gcmu\gcl{1}\gcl{1}\gnl
\gcl{1}\gbr\gcl{1}\gnl
\gmu\gcl{1}\gcl{1}\gnl
\got{2}{\micro{H}}\got{1}{\micro{H}}\got{1}{\micro{X}}
\gend
=
\gbeg{6}{11}
\got{1}{\micro{X}}\gvac{5}\gnl
\gcl{1}\gvac{1}\gwdb{4}\gnl
\gcl{1}\glcm\gvac{1}\glcm\gnl
\gbr\gcn{2}{1}{1}{3}\gcl{1}\gcl{1}\gnl
\gcl{1}\gcl{1}\glcm\gcl{1}\gcl{1}\gnl
\gcl{1}\gbr\gcl{1}\gcl{1}\gcl{1}\gnl
\gcl{1}\gcl{1}\gev\gcl{1}\gcl{1}\gnl
\gcl{1}\gcl{1}\gcn{3}{1}{5}{1}\gcl{1}\gnl
\gcl{1}\gbr\gvac{2}\gcl{1}\gnl
\gmu\gcl{1}\gvac{2}\gcl{1}\gnl
\got{2}{\micro{H}}\got{1}{\micro{H}}\gvac{2}\got{1}{\micro{X}}
\gend
= 
\gbeg{5}{10}
\got{1}{\micro{X}}\gvac{5}\gnl
\gcl{1}\gvac{1}\gwdb{3}\gnl
\gcl{1}\glcm\glcm\gnl
\gcl{1}\gcl{1}\gbr\gcl{1}\gnl
\gcl{1}\gmu\gcl{1}\gcl{1}\gnl
\gcl{1}\gcn{1}{1}{2}{1}\glcm\gcl{1}\gnl
\gbr\gcl{1}\gcl{1}\gcl{1}\gnl
\gcl{1}\gbr\gcl{1}\gcl{1}\gnl
\gcl{1}\gcl{1}\gev\gcl{1}\gnl
\got{1}{\micro{H}}\got{1}{\micro{H}}\gvac{2}\got{1}{\micro{X}}
\gend
\stackrel{\eqref{eq:dbcolin}}{=} 
\gbeg{5}{6}
\got{1}{\micro{X}}\gvac{5}\gnl
\gcl{1}\gu{1}\gvac{1}\gdb\gnl
\gbr\glcm\gcl{1}\gnl
\gcl{1}\gbr\gcl{1}\gcl{1}\gnl
\gcl{1}\gcl{1}\gev\gcl{1}\gnl
\got{1}{\micro{H}}\got{1}{\micro{H}}\gvac{2}\got{1}{\micro{X}}
\gend
\end{equation*}
i.e. for every $h\in H$ we have $\sum S\left( h_{1}\right)_{1}h_{2}\otimes S\left( h_{1}\right) _{2}=1_{H}\otimes S\left( h\right) $. Now, since $\ev{\forget(X)}$ is $H$-colinear as well, we have also
\begin{equation*}
\gbeg{4}{10}
\gvac{3}\got{1}{\micro{X}}\gnl
\gvac{2}\glcm\gnl
\gvac{1}\gcn{2}{1}{3}{1}\gcl{1}\gnl
\gwcm{3}\gcl{1}\gnl
\gcl{1}\gvac{1}\gmp{S}\gcl{1}\gnl
\gcl{1}\gcn{2}{1}{3}{2}\gcl{1}\gnl
\gcl{1}\gcmu\gcl{1}\gnl
\gbr\gcl{1}\gcl{1}\gnl
\gcl{1}\gmu\gcl{1}\gnl
\got{1}{\micro{H}}\got{2}{\micro{H}}\got{1}{\micro{X}}
\gend
\stackrel{\eqref{eq:preantnabla}}{=}
\gbeg{5}{9}
\gvac{1}\got{1}{\micro{X}}\gvac{3}\gnl
\glcm\gvac{1}\gdb\gnl
\gcl{1}\gcl{1}\glcm\gcl{1}\gnl
\gcl{1}\gbr\gcl{1}\gcl{1}\gnl
\gcl{1}\gcn{1}{1}{1}{2}\gev\gcl{1}\gnl
\gcl{1}\gcmu\gvac{1}\gcl{1}\gnl
\gbr\gcl{1}\gvac{1}\gcl{1}\gnl
\gcl{1}\gmu\gvac{1}\gcl{1}\gnl
\got{1}{\micro{H}}\got{2}{\micro{H}}\gvac{1}\got{1}{\micro{X}}
\gend
=
\gbeg{6}{9}
\gvac{1}\got{1}{\micro{X}}\gvac{3}\gnl
\gvac{1}\gcl{1}\gvac{2}\gdb\gnl
\glcm\gvac{1}\glcm\gcl{1}\gnl
\gcl{1}\gcl{1}\gcn{2}{1}{3}{1}\gcl{1}\gcl{1}\gnl
\gcl{1}\gbr\glcm\gcl{1}\gnl
\gcl{1}\gcl{1}\gbr\gcl{1}\gcl{1}\gnl
\gbr\gcl{1}\gev\gcl{1}\gnl
\gcl{1}\gmu\gvac{2}\gcl{1}\gnl
\got{1}{\micro{H}}\got{2}{\micro{H}}\gvac{2}\got{1}{\micro{X}}
\gend
=
\gbeg{6}{9}
\gvac{1}\got{1}{\micro{X}}\gvac{4}\gnl
\gvac{1}\gcl{1}\gvac{2}\gdb\gnl
\gvac{1}\gcl{1}\gvac{1}\glcm\gcl{1}\gnl
\gvac{1}\gcl{1}\gcn{2}{1}{3}{1}\gcl{1}\gcl{1}\gnl
\gvac{1}\gbr\gvac{1}\gcl{1}\gcl{1}\gnl
\gcn{1}{1}{3}{1}\glcm\glcm\gcl{1}\gnl
\gcl{1}\gcl{1}\gbr\gcl{1}\gcl{1}\gnl
\gcl{1}\gmu\gev\gcl{1}\gnl
\got{1}{\micro{H}}\got{2}{\micro{H}}\gvac{2}\got{1}{\micro{X}}
\gend
\stackrel{\eqref{eq:dbcolin}}{=}
\gbeg{4}{7}
\got{1}{\micro{X}}\gvac{2}\gnl
\gcl{1}\gvac{1}\gdb\gnl
\gcl{1}\glcm\gcl{1}\gnl
\gbr\gcl{1}\gcl{1}\gnl
\gcl{1}\gev\gcl{1}\gnl
\gcl{1}\gu{1}\gvac{1}\gcl{1}\gnl
\got{1}{\micro{H}}\got{1}{\micro{H}}\gvac{1}\got{1}{\micro{X}}
\gend
\end{equation*}
i.e. for every $h\in H$ we have $\sum S\left( h_{2}\right)_{1}\otimes h_{1}S\left( h_{2}\right) _{2}=S\left( h\right) \otimes 1_{H}$. Finally
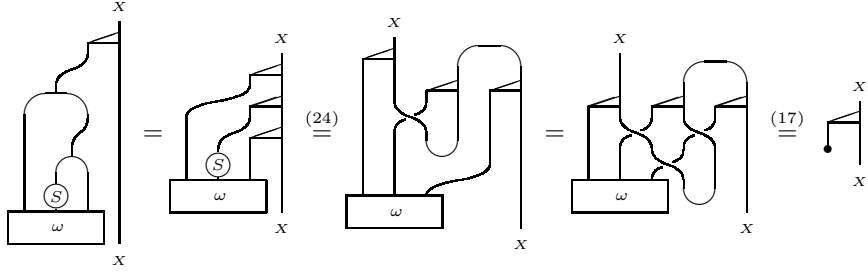
\begin{equation*}
\gbeg{4}{9}
\gvac{3}\got{1}{\micro{X}}\gnl
\gvac{2}\glcm\gnl
\gcn{3}{1}{5}{3}\gcl{1}\gnl
\gwcm{3}\gcl{1}\gnl
\gcl{1}\gcn{2}{1}{3}{2}\gcl{1}\gnl
\gcl{1}\gcmu\gcl{1}\gnl
\gcl{1}\gmp{S}\gcl{1}\gcl{1}\gnl
\gsbox{3}\gnot{\hspace{9mm}\omega}\gvac{3}\gcl{1}\gnl
\gvac{3}\got{1}{\micro{X}}
\gend
=
\gbeg{4}{7}
\gvac{3}\got{1}{\micro{X}}\gnl
\gvac{2}\glcm\gnl
\gcn{2}{1}{5}{1}\glcm\gnl
\gcl{1}\gcn{1}{1}{3}{1}\glcm\gnl
\gcl{1}\gmp{S}\gcl{1}\gcl{1}\gnl
\gsbox{3}\gnot{\hspace{9mm}\omega}\gvac{3}\gcl{1}\gnl
\gvac{3}\got{1}{\micro{X}}
\gend
\stackrel{\eqref{eq:preantnabla}}{=}
\gbeg{6}{8}
\gvac{1}\got{1}{\micro{X}}\gvac{4}\gnl
\glcm\gvac{1}\gwdb{3}\gnl
\gcl{1}\gcl{1}\glcm\glcm\gnl
\gcl{1}\gbr\gcl{1}\gcl{1}\gcl{1}\gnl
\gcl{1}\gcl{1}\gev\gcl{1}\gcl{1}\gnl
\gcl{1}\gcl{1}\gcn{3}{1}{5}{1}\gcl{1}\gnl
\gsbox{3}\gnot{\hspace{9mm}\omega}\gvac{3}\gvac{2}\gcl{1}\gnl
\gvac{5}\got{1}{\micro{X}}
\gend
=
\gbeg{6}{7}
\gvac{1}\got{1}{\micro{X}}\gvac{4}\gnl
\gvac{1}\gcl{1}\gvac{1}\gwdb{3}\gnl
\glcm\glcm\glcm\gnl
\gcl{1}\gbr\gbr\gcl{1}\gnl
\gcl{1}\gcl{1}\gbr\gcl{1}\gcl{1}\gnl
\gsbox{3}\gnot{\hspace{9mm}\omega}\gvac{3}\gev\gcl{1}\gnl
\gvac{5}\got{1}{\micro{X}}
\gend
\stackrel{\eqref{eq:Xdual}}{=}
\gbeg{2}{4}
\gvac{1}\got{1}{\micro{X}}\gnl
\glcm\gnl
\gcu{1}\gcl{1}\gnl
\gvac{1}\got{1}{\micro{X}}
\gend
\end{equation*}
so that $\sum \omega \left(h_{1}\otimes S\left( h_{2}\right) \otimes h_{3}\right) = \varepsilon(h) $ for all $h\in H$.
\end{proof}

\noindent Summing up, we can state our main theorem, connecting the rigidity of the category $\cC $ with the existence of a preantipode for the coendomorphism coquasi-bialgebra.

\begin{theorem}\label{th:reconstructionpreantipode}
Let $\cC$ be an essentially small right rigid monoidal category together with a neutral quasi-monoidal functor $\forget:\cC\to\fvect $. Then there exists a preantipode $S$ for the coendomorphism coquasi-bialgebra $H$ of $\left(\cC,\forget\right)$.
\end{theorem}

\begin{corollary}[{\cite[page 255, Theorem]{Ul}}] If in addition $\forget:\cC \rightarrow { \M }_{f}$ is monoidal, then the coendomorphism coquasi-bialgebra $H$ is a Hopf algebra.
\end{corollary}

\begin{proof}
Since $\omega =\varepsilon \otimes \varepsilon \otimes \varepsilon $, $H$ is a bialgebra and the preantipode provided by Theorem \ref{th:reconstructionpreantipode} satisfies
\begin{equation*}
\varepsilon S\left( h\right) =\sum \omega \left( h_{1}\otimes S\left(h_{2}\right) \otimes h_{3}\right) \stackrel{\eqref{fond S}}{=}\varepsilon \left( h\right) ,
\end{equation*} 
\ie it is an ordinary antipode (see Remark \ref{Lempreant}).
\end{proof}

\begin{remark}
Between the distinguished natural transformations in $\Nat(\forget,\forget)$ that one may consider, there is also $\left(\forget(X) \otimes \db{\forget(X)}\right)\left( \ev{\forget(X)} \otimes \forget(X) \right)$. This however does not endow $H$ with a new structure. Instead, it can be checked that
\begin{align*}
\left( \varepsilon S \otimes \forget(X) \right) \frho_{X} & = \left(\forget(X) \otimes \db{\forget(X)}\right)\left( \ev{\forget(X)} \otimes \forget(X) \right) \\
 & = \left(\left( \omega^{-1}(S\otimes H \otimes S)(\Delta \otimes H)\Delta \right)\otimes \forget(X)\right)\delta_X
\end{align*}
whence $\omega^{-1}(S(h_1)\otimes h_2\otimes S(h_3)) = \varepsilon S(h)$ for all $h\in H$ as in \cite[Lemma 2.14]{Ardi-Pava-Bosonization}.
\end{remark}

We conclude this subsection by showing that Theorem \ref{th:reconstructionpreantipode} can be refined in order to get a reconstruction theorem for coquasi-Hopf algebras as well (this result can be considered as a dual version of \cite[Lemma 4]{reconQuasiHopf}). Recall that a \emph{coquasi-Hopf algebra} is a coquasi-bialgebra $H$ endowed with a \emph{coquasi-antipode}, that is to say, a triple $(s,\alpha,\beta)$ composed by an anti-coalgebra endomorphism $s:H\rightarrow H$ and two linear maps $\alpha ,\beta \in H^{\ast }$, such that, for all $h\in H$
\begin{gather*}
\sum h_{1}\beta (h_{2})s(h_{3}) =\beta (h)1_{H}, \quad \sum s(h_{1})\alpha (h_{2})h_{3} =\alpha (h)1_{H}, \\
\sum \omega (h_{1}\otimes \beta (h_{2})s(h_{3})\alpha (h_{4})\otimes h_{5})
 = \varepsilon (h), \\
\sum \omega ^{-1}(s(h_{1})\otimes \alpha
(h_{2})h_{3}\beta (h_{4})\otimes s(h_{5}))= \varepsilon (h).
\end{gather*}

\begin{remark}[Reconstruction theorem for coquasi-Hopf algebras]\label{rem:reconcoquasi}
Let $(\cC,\forget)$ be as in Theorem \ref{th:reconstructionpreantipode} and assume in addition that we have a natural isomorphism $\nu_{X}:\forget\left(X^\star\right)\rightarrow \forget(X)^*$ in $\M_f$. Consider the associated coendomorphism coquasi-bialgebra $H$.  We may endow $\forget(X)^*$ with an $H$-comodule structure given by $\rho_{\forget(X)^*}:=\nu_X\circ \delta_{X^\star} \circ \nu_X^{-1}$. To simplify the exposition, we will denote it by $\frho_{X^*}$, even if this notation does not strictly make sense. With this coaction, $\forget(X)^*$ becomes a right dual object of $\forget(X)$ in ${^H\M_f}$ with evaluation and dual basis maps given by
\begin{equation*}
\ev{\forget(X)}^{(*)} = \ev{\forget(X)}  \left(\forget(X)\otimes \nu_X^{-1} \right) \qquad \text{and} \qquad \db{\forget(X)}^{(*)} = \left(\nu_X\otimes \forget(X)\right)  \db{\forget(X)}.
\end{equation*}
Denote by $\cV:{^H\M_f}\rightarrow \M_f$ the forgetful functor and by $\ev{V}^{(\K)}:V\otimes V^*\rightarrow \K$ and $\db{V}^{(\K)}:\K\rightarrow V^*\otimes V$ the ordinary evaluation and dual basis maps for finitely-generated and projective $\K$-modules. Then, there exist unique linear morphisms $\alpha,\beta\in H^*$ and $s:H\rightarrow H$ such that
\begin{gather*}
\left( \alpha \otimes \forget(X) \right) \delta_X = \left( \ev{\forget(X)}^{(\K)} \otimes \forget(X) \right) \left( \forget(X) \otimes \db{\forget(X)}^{(*)} \right), \\
\left( \beta \otimes \forget(X) \right) \delta_X = \left( \ev{\forget(X)}^{(*)} \otimes \forget(X) \right) \left( \forget(X) \otimes \db{\forget(X)}^{(\K)} \right), \\
\left( s \otimes \forget(X) \right) \delta_X = \left( H \otimes \ev{\forget(X)}^{(\K)} \otimes \forget(X) \right) \tau_{\forget(X),H}  \delta_{X^*} \left( \forget(X) \otimes \db{\forget(X)}^{(\K)} \right).
\end{gather*}
It can be checked that $(s,\alpha,\beta)$ is a coquasi-antipode for $H$. A posteriori, notice that the following relations hold
\begin{subequations}\label{eq:evdbcoquasiHopf}
\begin{align}
\tau_{\forget(X)^*,H} \left( \forget(X)^* \otimes H \otimes \ev{\forget(X)}^{(\K)} \right) s \, \delta_X ( \db{\forget(X)}^{(\K)} \otimes \forget & (X)^* ) = \delta_{X^*} \label{eq:rhostarcoquasiHopf}, \\
\left(  \forget(X)^* \otimes  \alpha \otimes \forget(X) \right) \left( \forget(X)^* \otimes \delta_X \right) \db{\forget(X)}^{(\K)} & = \db{\forget(X)}^{(*)} , \label{eq:dbcoquasihopf}\\
\ev{\forget(X)}^{(\K)} \left( \beta \otimes \forget(X) \otimes \forget(X)^* \right) \left( \delta_X \otimes \forget(X)^* \right) & = \ev{\forget(X)}^{(*)} \label{eq:evcoquasihopf}.
\end{align}
\end{subequations}
\end{remark}

\subsection{The field case}\label{ssec:field}
Under the additional assumption that $\K$ is a field, some previous results can be refined and some further conclusions can be drawn.
The key point, as we already mentioned in Remark \ref{rem:reconstructioncoquasi}, is that the reconstruction process applied to ${^C\fvect}$ for $C$ a coalgebra over a field gives back the starting coalgebra.

\begin{remark}\label{rem:compnabla}
Assume that $B$ is a coquasi-bialgebra with a preantipode $S_{ {B}}$. Denote by $\cU:{^{B}\M_f}\rightarrow\M_f$ the forgetful functor and by $\rho\in\Nat \left(\cU,B\otimes\cU\right)$ the natural coaction of the $B$-comodules in ${^{B}\M_f}$. It can be checked that for $V$ in ${^{B}\M_f}$, a right dual of $V$ is given by $V^{\star}=\left(V^*\otimes B\right)^{\mathrm{co}B}$ with coaction $\rho _{ {V^{\star  }}}\left( \sum_{t}f_{t}\otimes b_{t}\right)=\sum_{t}\left(b_{t}\right) _{1}\otimes \left( f_{t}\otimes \left(b_{t}\right) _{2}\right)$. Evaluation and dual basis maps are given by
\begin{gather*}
\ev{V} \hspace{-1pt} \left(u\otimes \sum_{t}( f_{t}\otimes b_{t})\right) \hspace{-1pt} =\sum_{t}f_{t}( u) \varepsilon( b_{t}), \qquad 
\db{V}(1_{\Bbbk})= \sum_{i=1}^{ d_{ {V}}}\left( v^{i} _{0}\otimes S_{ {B}}( v^{i}_{1})\right) \otimes v_{i},
\end{gather*}
for all $\sum_{t}f_{t}\otimes b_{t}\in V^{\star}$, $u\in V$, where $\sum_{i=1}^{ d_{ {V}}}v^i\otimes v_i\in V^*\otimes V$ is a dual basis for $V$ as a finite-dimensional vector space. In particular, ${^B}\M_f$ is right rigid.
\begin{invisible}
Notice that the reason why we need $\K$ to be a field in this case is that we want a comodule structure on $\left(V^*\otimes B\right)^{\mathrm{co}B}$.
\end{invisible}
\end{remark}

The following theorem characterizes coquasi-bialgebras with preantipode and coquasi-Hopf algebras via rigidity of the categories of finite-dimensional comodules.

\begin{theorem}\label{thm:characterization}
A coquasi-bialgebra $B$ admits a preantipode if and only if ${^B\fvect}$ is right rigid. It is a coquasi-Hopf algebra if and only if, in addition, the right duals can be given by endowing the $\K$-linear duals with a suitable comodule structure.
\end{theorem}

\begin{proof}
The first assertion follows from Theorem \ref{th:reconstructionpreantipode} together with Remark \ref{rem:compnabla}. For the second one, it is already known that the converse of Remark \ref{rem:reconcoquasi} holds true, in the sense that the category of finite-dimensional left $H$-comodules ${^{H} \M }_{f}$ over a coquasi-Hopf algebra $H$ is a right rigid monoidal category, where the dual of $\left( V,\rho_V \right) $ in ${^{H} \M }_{f}$ is given by its dual vector space $V^{\ast }$ with structure maps given as in \eqref{eq:evdbcoquasiHopf} (cf. \cite[page 334]{Schauenburg-HopExt}).
\begin{invisible}
Explicitly, the dual of $\left( V,\rho_V \right) $ in ${^{H} \M }_{f}$ is given by its dual vector space $V^{\ast }$ with comodule structure $\rho _{V^\ast }\left( f\right) =\sum f_{-1}\otimes f_{0}$ given by the relation
\begin{equation}
\sum f_{-1}f_{0}\left( v\right) =\sum s\left( v_{-1}\right) f\left(
v_{0}\right)  \label{eq:coactant}
\end{equation}
for all $v\in V$, $f\in V^{\ast }$. Moreover, if $\db{V}^{(\Bbbk) }\left( 1_{\Bbbk}\right) =\sum_{i=1}^{ d_{ {V}} }v^{i}\otimes v_{i}\in V^{\ast }\otimes V$ is a dual basis for $V$ as a vector space, then 
\begin{equation}
\rho _{V^\ast }\left( f\right) =\sum_{i=1}^{ d_{ {V}} 
}s\left( \left( v_{i}\right) _{-1}\right) f\left( \left( v_{i}\right)
_{0}\right) \otimes v^{i}
\end{equation} 
and the evaluation and dual basis morphisms are given by
\begin{equation}
\ev{V}\left( v\otimes f\right) =\sum \beta \left( v_{-1}\right) f\left( v_{0}\right) , \qquad \db{V}\left( 1_{\Bbbk }\right) =\sum_{i=1}^{ d_{ {V}}}v^{i}\otimes \alpha \left( \left( v_{i}\right) _{-1}\right) \left( v_{i}\right) _{0}
\end{equation}
for all $v\in V$, $f\in V^{\ast }$ (cf. \cite[page 334]{Schauenburg-HopExt}). 
\end{invisible}
\end{proof}

\begin{corollary}[{\cite[Theorem 3.10]{Ardi-Pava}}]\label{cor:quasi-preantipode}
Every coquasi-Hopf algebra $H$ with coquasi-antipode $\left(s,\alpha ,\beta \right) $ admits a preantipode given by $S:=\beta*s*\alpha$.
\end{corollary}

\begin{proof}
We may apply Theorem \ref{th:reconstructionpreantipode} as well as Remark \ref{rem:reconcoquasi} with $\nu_V=\id_{V^*}$ for all $(V,\rho_V)$ in ${^H\M_f}$. The outcome is a coquasi-bialgebra with preantipode structure on $H$, where $\omega $ is the former one and $S $ is uniquely given by 
\begin{align*}
( S  &  \otimes V)\rho_V \stackrel{\eqref{eq:preantnabla}}{=} (H\otimes \ev{V}^{(*)} \otimes V)\, \tau_{{V,H}}\,  \frho_{V^*} \, (V\otimes \db{V}^{(*)}) \\
 & \stackrel{\eqref{eq:evdbcoquasiHopf}}{=} \ev{V}^{(\K)}  \left( H \otimes (\beta \otimes V)\rho_V \otimes V^* \otimes V \right) \tau_{{V,H}}\,  \frho_{V^*} \left(V \otimes   V^* \otimes \left( \alpha \otimes V \right) \rho_V \right) \left( V \otimes\db{V}^{(\K)} \right) \\
 & \stackrel{\phantom{(27)}}{=} \left( \alpha \otimes V \right) \rho_V \, \ev{V}^{(\K)} \,  \tau_{{V,H}}\,  \frho_{V^*} \left( V \otimes\db{V}^{(\K)} \right) (\beta \otimes V)\rho_V \\
 & \stackrel{\eqref{eq:rhostarcoquasiHopf}}{=} \left( H \otimes \alpha \otimes V \right) \rho_V \left( s \otimes V\right) \rho_V (\beta \otimes V)\rho_V = ((\beta * s* \alpha) \otimes V)\rho_V
\end{align*}
Therefore, $S=\beta \ast s\ast \alpha $.
\end{proof}

\begin{lemma}
If a preantipode for a coquasi-bialgebra $B$ exists, it is unique.
\end{lemma}

\begin{proof}
It can be checked directly that $S_{ {B}}$ satisfies condition \eqref{eq:preantnabla} and hence $S_{ {B}}=S$, the unique linear endomorphism induced on $B$.
\begin{invisible}
In details, since ${^B}\M_f$ is right rigid, $\cU:{^{B}\M_f}\rightarrow\M_f$ is a tensor functor and $B$ is a representing object for $\Nat(\cU,-\otimes \cU)$, we have the natural transformation $\nabla^{{\cU}}_B(\rho)\in\Nat \left(\cU,B\otimes \cU\right)$ as in \eqref{eq:nabla} (in fact, $\rho=\vartheta_{B}(\id_{B})$). In view of Remark \ref{rem:compnabla}, we may compute explicitly for all $V$ in ${^{B}\M_f}$ and $y\in \cU(V)$
\begin{align*}
& \nabla^{{\cU}}_B(\rho)_{{V}}(y)\stackrel{\eqref{eq:nabla}}{=}\left( \left( B\otimes \ev{V}\otimes V\right) \, \tau_{ {V,B}}\, \rho_{ {V^{\star } }}\right) \left( \sum_{i=1}^{ d_{ {V}}}y\otimes \left(  v^{i} _{0}\otimes S_{ {B}}\left(v^{i}_{1}\right)\right) \otimes
v_{i}\right) \\
&\stackrel{\phantom{(31)}}{=}\sum_{i=1}^{ d_{ {V}}}S_{ {B}}\left( v^{i} _{1}\right)_{1} v^{i}_{0}\left(
y\right) \varepsilon\left( S_{ {B}}\left( v^{i} _{1}\right)_{2}\right)
\otimes v_{i} =\sum_{i=1}^{ d_{ {V}}}S_{ {B}}\left(  v^{i}_{0}\left(y\right) v^{i}_{1} \right)\otimes v_{i} \\
 &\stackrel{\phantom{(31)}}{=}\sum_{i=1}^{ d_{ {V}}}S_{ {B}}(y_{-1}v^{i}\left( y_0\right)
)\otimes v_{i}=\sum S_{ {B}}\left( y_{-1}\right) \otimes y_{0},
\end{align*}
so that $\nabla^{{\cU}}_B(\rho)=\left(S_{{B}}\otimes \cU\right) \rho$. This means that $S_{ {B}}$ satisfies condition \eqref{eq:preantnabla} and so it follows that $S_{ {B}}=S$, the unique linear endomorphism induced on $B$ by $\nabla^{{\cU}}_B(\rho)$.
\end{invisible}
\end{proof}

\begin{lemma}\label{lemma:preantipodemorphi}
Let $g:A\rightarrow B$ be a morphism between coquasi-bialgebras $A$ and $B$ with preantipodes $S_{ {A}}$ and $S_{ {B}}$ respectively. Then $g \, S_{ {A}}=S_{ {B}}\, g$.
\end{lemma}

\begin{proof}
Since $g$ is a coquasi-bialgebra morphism, it induces a strict monoidal functor ${^{g} \M }:{^{A} \M }\rightarrow {^{B} \M }$, which in turn restricts to a strict monoidal functor $\mathcal{G}:{^{A}\fvect }\rightarrow {^{B}\fvect }$ such that $\cV\cG=\cU$, where $\cU:{^{A}\fvect }\rightarrow {\fvect }$ and $\cV:{^{B}\fvect }\rightarrow {\fvect }$ are the forgetful functors. Observe that, in particular, this implies that $\left(g\otimes \cU(X)\right)\,\rho_{ {X}}^{ {A}}=\rho^{ {B}}_{ {\cG(X)}}$ for every $X$ in ${^A\M_f}$. Let us denote by $\vartheta:\Homk(A,-)\rightarrow \Nat \left(\cU,-\otimes \cU\right)$ the natural isomorphism such that $\vartheta_{ {V}}(f)=\left(f\otimes \cU\right)\, \rho^{ {A}}$ for all $V$ in $\M$ and $f\in\Homk(A,V)$. We want to show that $\vartheta_{ {B}}(g\, S_{ {A}})=\vartheta_{ {B}}(S_{ {B}}\, g)$. For all $X$ in ${^{A}\fvect }$ we may compute
\begin{align*}
\vartheta _{{B}} & ( g  S_{{A}})_{X} \stackrel{\eqref{eq:thetanabla}}{=} \nabla_{B}^{\cU}\left((g\otimes \cU)\,\rho^A\right)_X = \nabla_B^{\cV\cG}\left(\rho^B\cG\right)_X \stackrel{\eqref{eq:Nablacomp}}{=} \nabla_B^\cV\left(\rho^B\right)_{\cG(X)}\\
& = \left( S_B\otimes \cV\cG(X)\right) \rho^B_{\cG(X)}= \left( S_B\otimes \cU(X)\right) (g\otimes \cU(X)) \rho^A_{X} = \vartheta_B\left( S_B\, g \right)_X.
\end{align*}
Hence $g\, S_{ {A}}=S_{ {B}}\, g$ as claimed.
\end{proof}

\begin{remark}
We conclude this subsection with two observations. 
\begin{enumerate}
\item Let $B$ be a coquasi-bialgebra over a commutative ring $\K$. The natural coaction $\rho\in\Nat\left(\cU,B\otimes \cU\right)$ coming from ${^B\fvect}$ induces, via the isomorphism $\Homk(H_{\cU},B)\cong \Nat\left(\cU,B\otimes \cU\right)$, a canonical morphism $\mathsf{can}_B:H_{\cU}\to B$. Then, all the results in \S\ref{ssec:field} still remain true if we assume to work with $\K$-flat coquasi-bialgebras $B$ such that $\mathsf{can}_B$ is an isomorphism (mimicking \cite{LaiachiPepe,PepeJoost}, these may be referred to as \emph{Galois coalgebras}).
\item Even if the characterization of coquasi-bialgebras with preantipode in Theorem \ref{thm:characterization} was already known before (see \cite{Sch-TwoChar}), the one for coquasi-Hopf algebras is a new achievement, as far as we know.
\end{enumerate}
\end{remark}

\section{The finite-dual of a quasi-bialgebra with preantipode}\label{sec:finitedual}

Assume that $\K$ is a field. As a final application of the theory we developed, let us show that the finite dual coalgebra of a quasi-bialgebra with preantipode is a coquasi-bialgebra with preantipode (for the definition of the finite dual coalgebra we refer to \cite[Chapter VI]{Swe}, for the definition of a preantipode for a quasi-bialgebra and its properties we refer to \cite{StructTheo}). The proof of this fact lies on the following result, which can be deduced from \cite[Chapter 3, \S1.2]{Abe}.

\begin{lemma}\label{lemma:isocat}
Let $A$ be an algebra and $A^\circ$ be its finite dual coalgebra. We have an isomorphism $\cL:{^{A^\circ}\M_f}\rightarrow {_f\M_A}$ between the category of finite-dimensional left $A^\circ$-comodules and that of finite-dimensional right $A$-modules that satisfies $\cV\,\cL=\cU$, where $\cV:{_f\M_A}\rightarrow\M_f$ and $\cU:{^{A^\circ}\M_f}\rightarrow\M_f$ are the obvious forgetful functors.
\end{lemma}

For the sake of completeness, let us recall that $\cL$ associates every  left $A^\circ$-comodule $(N,\rho_{ {N}})$ with the right $A$-module $\left(N,\mu^{ {\rho}}_{ {N}}\right)$ where the action is given by $\mu_{ {N}}^{ {\rho}}\left(n\otimes a\right)=\sum n_{-1}(a)n_0.$
Its inverse $\cR:{_f\M_{A}}\rightarrow{^{A^\circ}\M_f}$ assigns to every finite-dimensional right $A$-module $(M,\mu_{ {M}})$, the left $A^\circ$-comodule $(M,\rho^{ {\mu}}_{ {M}})$ with coaction
\begin{equation}\label{eq:coaction}
\rho_{ {M}}^{ {\mu}}\left(m\right)=\sum_{i=1}^{ d_{ {M}}}\left(e^i \mu_m\right)\otimes e_i 
\end{equation}
where $\mu_m(a):=\mu_{ {M}}(m\otimes a)$ for all $a\in A$ and $\sum_{i=1}^{ d_{ {M}}}e^i\otimes e_i\in M^\ast\otimes M$ is a dual basis for $M$ as a vector space. Notice that ${\cU}\cR=\cV$ as well.

\begin{lemma}\label{lemma:rigidmod}
Let $\left(A,m,u,\Delta,\varepsilon,\Phi,S\right)$ be a quasi-bialgebra with preantipode. The category of finite-dimensional right $A$-modules ${_f\M_A}$ is a right rigid monoidal category with quasi-monoidal forgetful functor $\cV:{_f\M_A}\rightarrow \M_f$.
\end{lemma}

\begin{proof}
As it happens for coquasi-bialgebras, the axioms of a quasi-bialgebra ensures that the category of right $A$-modules $\M_A$ becomes a monoidal category with tensor product the tensor product over $\Bbbk$, unit object $\Bbbk$ itself, and associativity constraint
\begin{equation*}
\calpha_{ {M,N,P}}\left(\left(m\otimes n\right)\otimes p\right)=\left(m\otimes \left(n\otimes p\right)\right)\cdot \Phi^{-1}
\end{equation*}
for all $M$, $N$, $P$ in ${\M_A}$ and $m$, $n$, $p$ in $M$, $N$, $P$ respectively.
The unit constraints are the same of $\M$. In particular, the forgetful functor $\cV:\M_A\rightarrow \M$ is a quasi-monoidal functor and the same property holds for its restriction to finite-dimensional modules.
One may check directly that a dual object of an $A$-module $M$ is given by 
\begin{equation*}
M^\star:=\frac{A\otimes M^\ast}{A^+\left(A\otimes M^\ast\right)}
\end{equation*}
where $A^+:=\ker\left(\varepsilon\right)$ and $M^\ast$ is the $\Bbbk$-linear dual of $M$. The $A$-module structure on $M^\star$ is $\overline{a\otimes f}\cdot x=\overline{ax\otimes f}$ for all $a,x\in A$ and $f\in M^\ast$. Evaluation and dual basis maps are explicitly given by
\begin{equation*}
\ev{M}\left(m\otimes \overline{a\otimes f}\right)=f\left(m\cdot S(a)\right) \qquad\text{and}\qquad \db{M}\left(1_{\Bbbk}\right)=\sum_{i=1}^{ d_{ {M}}}\overline{1_A\otimes e^i}\otimes e_i
\end{equation*}
for all $m\in M$, $f\in M^\ast$ and  $a\in A$ and where $\sum_{i=1}^{ d_{ {M}}} e^i\otimes e_i$ is a dual basis of $M$ as a finite-dimensional vector space.
\end{proof}

\begin{proposition}\label{prop:finitedual}
Let $\left(A,m,u,\Delta,\varepsilon,\Phi,S\right)$ be a quasi-bialgebra with preantipode. Let $\left(A^\circ,\Delta_\circ,\varepsilon_\circ\right)$ be its finite dual coalgebra. Then $A^\circ$ can be endowed with a structure of a coquasi-bialgebra with preantipode.
\end{proposition}

\begin{proof}
Denote by $\cV:{_f\M_A}\rightarrow\fvect$ and $\cU:{^{A^\circ}\M_f}\rightarrow\fvect$ the forgetful functors. As a consequence of Lemma \ref{lemma:isocat}, we have a chain of natural isomorphism
\begin{equation*}
\mathrm{Nat}\left(\cV,-\otimes \cV\right)\cong \mathrm{Nat}\left(\cU,-\otimes \cU\right)\cong \Homk\left(A^\circ,-\right)
\end{equation*}
which allows us to consider $A^\circ$ itself as a representing object for $\mathrm{Nat}\left(\cV,-\otimes \cV\right)$.
If we consider then the category of finite-dimensional right $A$-modules ${_f\M_A}$ as a right rigid monoidal category together with the quasi-monoidal forgetful functor $\cV:{_f\M_A}\rightarrow \M_f$, then $A^\circ$ can be endowed with a structure of a coquasi-bialgebra with preantipode in view of Theorem \ref{th:reconstructionpreantipode}.
\end{proof}

\begin{remark}
It is worthy to point out that the corestriction $\cV^{A^\circ}:{_f\M_A}\rightarrow{^{A^\circ}\M_f}$ of the functor $\cV^{A^\circ}:{_f\M_A}\rightarrow{^{A^\circ}\M}$ provided by Theorem \ref{th:reconstructionpreantipode} coincides with the functor $\cR$, which becomes this way a strict monoidal functor.
\end{remark}

\begin{remark}
If we want to know explicitly the coquasi-bialgebra structure on $A^\circ$ we may proceed as follows. First of all observe that the quasi-monoidal structure on $\cV:{_f\M_A}\rightarrow\fvect$ is the strict one: $\varphi_{ {M,N}}=\mathrm{id}_{M\otimes N}$ and $\varphi_0=\mathrm{id}_{\Bbbk}$. Secondly, for every object $M$ in ${_f\M_A}$ the natural transformation $\rho_{ {M}}:{\cV}(M)\rightarrow A^\circ\otimes {\cV}(M)$ is given by the coaction \eqref{eq:coaction}. Let us denote by $\sum_{i=1}^{d_{ {M}}}e^i_{ {M}}\otimes e^{ {M}}_i\in M^\ast\otimes M$ a dual basis for $M$ as a vector space, for all $M$ in ${_f\M_A}$. If we denote by $\mu^{ {M\otimes N}}$ the $A$-action on the tensor product, then
\begin{equation*}
\rho_{ {M\otimes N}}\left(x\right)=\sum_{i,j}\left(\left(e^i_{ {M}}\otimes e^j_{ {N}}\right)\, \mu^{ {M\otimes N}}_{x}\right)\otimes \left(e^{ {M}}_i\otimes e^{ {N}}_j\right)
\end{equation*}
for all $x\in M\otimes N$, where we considered $M^\ast\otimes N^\ast$ injected in $\left(M\otimes N\right)^\ast$. Furthermore, it is well-known from the associative case that the convolution product $*$ restricts to a morphism $*:A^\circ\otimes A^\circ \rightarrow A^\circ$. It is also clear that $\varepsilon\in A^\circ$. To show that they are the multiplication and the unit induced on $A^\circ$, denote by $\mu^{ {M}}$ and $\mu^{ {N}}$ the $A$-actions on $M$ and $N$ respectively and compute for $\sum_{i=1}^tm_i\otimes n_i\in M\otimes N$
\begin{equation*}
(A^\circ\otimes \varphi_{ {M,N}})\left(\vartheta_{ {A^\circ}}^2(*)_{ {M,N}}\left(\sum_{i=1}^tm_i\otimes n_i\right)\right) \hspace{-2pt} =\sum_{i,h,k}\left((e_{ {M}}^h \mu^{ {M}}_{m_i})*(e_{ {N}}^k \mu^{ {N}}_{n_i})\right)\otimes (e^{ {M}}_h\otimes e^{ {N}}_k).
\end{equation*}
Since for every $a\in A$, $f\in M^\ast$, $g\in N^\ast$ and $x=\sum_{i=1}^tm_i\otimes n_i\in M\otimes N$ we have
\begin{equation*}
\sum_{i=1}^t\left(\left(f\, \mu^{ {M}}_{m_i}\right)*\left(g\, \mu^{ {N}}_{n_i}\right)\right)(a)=\sum_{i=1}^t\left(f\, \mu^{ {M}}_{m_i}\right)(a_1)\left(g\, \mu^{ {N}}_{n_i}\right)(a_2)=\left(f\otimes g\right)\mu^{ {M\otimes N}}_{x}(a),
\end{equation*}
we conclude that $(A^\circ\otimes \varphi_{ {M,N}})\, \vartheta_{ {A^\circ}}^2(*)_{ {M,N}}=\rho_{ {M\otimes N}}\,\varphi_{ {M,N}}$ and by uniqueness of the morphism $A^\circ\otimes A^\circ\rightarrow A^\circ$ satisfying this relation we have that the multiplication induced on $A^\circ$ is exactly $*$. Moreover, if we compute
\begin{equation*}
r_{ {A^\circ}}\left(\rho_{\Bbbk}\left(1_{\Bbbk}\right)\right)=r_{ {A^\circ}}\left(\varepsilon\otimes 1_\Bbbk\right)=\varepsilon,
\end{equation*}
then we recover that the unit of the multiplication $*$ is $\varepsilon$, in view of \eqref{eq:defu} and the fact that $\varphi_0=\text{id}_{\Bbbk}$. Consider also the assignment
\begin{equation*}
\omega:A^\circ \otimes A^\circ \otimes A^\circ \rightarrow \Bbbk; \quad \omega\left(f\otimes g\otimes h\right)= \sum f\left(\Phi^1\right)g\left(\Phi^2\right)h\left(\Phi^3\right).
\end{equation*}
For every $M$, $N$, $P$ in ${_f\M_A}$ and all $m\in M$, $n\in N$, $p\in P$, it satisfies
\begin{align*}
& \varphi_{ {M\otimes N,P}}\left(\left(\varphi_{ {M,N}}\otimes \cV(P)\right)\left(\vartheta_{\Bbbk}^3\left(\omega\right)_{ {M,N,P}}(m\otimes n\otimes p)\right)\right) \\
 & =\sum_{i,j,k}\omega\left(\left(e_{ {M}}^i\,\mu_m^{ {M}}\right)\otimes \left(e_{ {N}}^j\, \mu_n^{ {N}}\right)\otimes \left(e_{ {P}}^k\,\mu_p^{ {P}}\right)\right)e^{ {M}}_i\otimes e^{ {N}}_j\otimes e^{ {P}}_k \\
& = \sum m\cdot \Phi^1\otimes n\cdot \Phi^2 \otimes p\cdot \Phi^3,
\end{align*}
whence $\varphi_{ {M\otimes N,P}} \, (\varphi_{ {M,N}}\otimes \cV(P) ) \, \vartheta_{\Bbbk}^3 (\omega )_{ {M,N,P}} = \cV (\calpha_{ {M,N,P}}^{-1} ) \, \varphi_{ {M,N\otimes P}} \, (\cV(M)\otimes \varphi_{ {N,P}} )
$ and so $\omega$ is in fact the induced reassociator. The antipode can be constructed explicitly as well. Consider the transpose $S^*:A^*\rightarrow A^*$. Let us show firstly that $S^\ast$ factors through a linear map $S^\circ:A^\circ\rightarrow A^\circ$; the proof relies on formula \eqref{antimultpreanti} from Appendix \ref{sec:appendix2}. Pick $f\in A^\circ$ and compute
\begin{equation*}
\begin{split}
& S^*\left(f\right)\left(ab\right) = f\left(S\left(ab\right)\right) \stackrel{\eqref{antimultpreanti}}{=} \sum f\left( S\left( \varphi ^{1}b\right) \varphi ^{2}S\left( \psi ^{1}\varphi ^{3}\right) \psi ^{2}S\left( a\psi ^{3}\right)\right) \\
 & = \sum f_1 S\left( \varphi ^{1}b\right)f_2\left( \varphi ^{2}S\left( \psi ^{1}\varphi ^{3}\right) \psi ^{2}\right) f_3S\left( a\psi ^{3}\right) \\
 & = \left(\sum \left(\psi ^{3}\rightharpoonup f_3S\right) \otimes f_2\left( \varphi ^{2}S\left( \psi ^{1}\varphi ^{3}\right) \psi ^{2}\right) \left(f_1 S\leftharpoonup\varphi ^{1}\right)\right) \left(a\otimes b\right).
\end{split}
\end{equation*}
Since this implies that $m^*\left(S^*\left(f\right)\right)\in A^\ast\otimes A^\ast$, in view of \cite[Proposition 6.0.3]{Swe} we have that $S^*\left(f\right)\in A^\circ$. Let us prove now that $S^\circ$ satisfies the relation $\vartheta_{ {A^\circ}}(S^\circ)=\nabla_{ {A^\circ}}^{ {\cV}}\left(\rho\right)$. For all $M$ in ${_f\M_A}$ and all $m\in M$ we need to show that
\begin{equation}\label{eq:condpreantipdoe}
\sum S^\circ\left(m_{-1}\right)\otimes m_0=\sum_{i=1}^{d_{ {M}}}\left(\overline{1_A\otimes e^i}\right)_0(m)\left(\overline{1_A\otimes e^i}\right)_{-1}\otimes e_i.
\end{equation}
Since $M^\star$ is finite-dimensional, we may fix a dual basis $\sum_{j=1}^{d_{ {M^\star}}}\gamma^j\otimes \gamma_j$ of $M^\star$ as an object in $\fvect$ and then, by \eqref{eq:coaction}, the right-hand member of \eqref{eq:condpreantipdoe} can be rewritten as
\begin{equation*}
\sum_{i=1}^{d_{ {M}}}\sum_{j=1}^{d_{ {M^\star}}}\gamma_j(m)\left(\gamma^j\mu_{\overline{1_A\otimes e^i}}^{ {M^\star}}\right)\otimes e_i.
\end{equation*}
Let us focus on $\sum_{j=1}^{d_{ {M^\star}}}\gamma_j(m)\left(\gamma^j\mu_{\overline{1_A\otimes e^i}}^{ {M^\star}}\right)\in A^\circ$. For all $a\in A$,
\begin{equation*}
\sum_{j=1}^{d_{ {M^\star}}}\gamma_j(m)\left(\gamma^j\mu_{\overline{1_A\otimes e^i}}^{ {M^\star}}\right)(a)=\sum_{j=1}^{d_{ {M^\star}}}\gamma_j(m)\gamma^j\left(\overline{a\otimes e^i}\right)=\overline{a\otimes e^i}(a) =e^i(m\cdot S(a))
\end{equation*}
and since $e^i(m\cdot S(a))=S^\circ\left(e^i\, \mu^{ {M}}_m\right)(a)$, we have
\begin{equation*}
\sum_{i=1}^{d_{ {M}}}\sum_{j=1}^{d_{ {M^\star}}}\gamma_j(m)\left(\gamma^j\mu_{\overline{1_A\otimes e^i}}^{ {M^\star}}\right)\otimes e_i =\sum_{i}S^\circ\left(e^i\, \mu^{ {M}}_m\right)\otimes e_i .
\end{equation*}
We can conclude then that relation \eqref{eq:condpreantipdoe} is satisfied, as desired.
\end{remark}

\begin{remark}
The fact that the finite dual coalgebra of a quasi-bialgebra is a coquasi-bialgebra has already been shown in \cite[\S5.2]{ArdiLaiachiPaolo} with a different approach.
\end{remark}

\appendix

\section{A relation for the preantipode of a quasi-bialgebra}\label{sec:appendix2}

Recall from \cite{StructTheo} that a preantipode for a quasi-bialgebra $(A,\Delta,\varepsilon,m,u,\Phi)$ is a $\K$-linear map $\varfun{S}{A}{A}$ that satisfies
\begin{gather}
\sum a_1S(ba_2)=\varepsilon(a)S(b)= \sum S(a_1b)a_2, \qquad \sum \Phi^1S(\Phi^2)\Phi^3=1, \label{eq:qbP}
\end{gather}
for all $a,b\in A$, where $\sum \Phi^1\otimes \Phi^2\otimes \Phi^3=\Phi$. Let us introduce also the following extended notation for the reassociator and its inverse:
\begin{gather*}
\Phi  =\sum \Phi ^{1}\otimes \Phi ^{2}\otimes \Phi ^{3}=\sum \Psi
^{1}\otimes \Psi ^{2}\otimes \Psi ^{3}=\ldots  \\
\Phi ^{-1} =\sum \phi ^{1}\otimes \phi ^{2}\otimes \phi ^{3}=\sum
\psi ^{1}\otimes \psi ^{2}\otimes \psi ^{3}=\ldots 
\end{gather*}
Let $(A,m,u,\Delta ,\varepsilon ,\Phi ,S)$ be a quasi-bialgebra with
preantipode and consider the $A$-actions on $\mathrm{End}(A)=\Homk(A,A)$  defined by $\left(f\leftharpoonup a\right)(b)  =  f(ab)$ and $\left(a\rightharpoonup f\right)(b)  =  f(ba)$ for all $a,b\in A$ and for all $f\in\End{}{A}$. Define the elements
\begin{eqnarray}
p :=\sum \varphi ^{1}\otimes \varphi ^{2}\left( \varphi ^{3}\rightharpoonup S\right) & \in & A\otimes \End{}{A} , \label{eq:Defp} \\
q :=\sum \left( S \leftharpoonup \varphi ^{1}\right) \varphi ^{2}\otimes \varphi ^{3} & \in & \End{}{A}\otimes A, \notag 
\end{eqnarray}
where $\left(x\left(y\rightharpoonup f\right)\right)(a)  =  xf(ay)$ and $\left(\left(f\leftharpoonup x\right)y\right)(a) = f(ax)y$ for all $a,x,y\in A$ and for all $f\in \End{}{A}$. Let us introduce the following notation for shortness:
\begin{equation*}
p:=\sum p^1\otimes p^2 \qquad \textrm{ and } \qquad q:=\sum q^1\otimes q^2.
\end{equation*}

\begin{lemma*}
\label{Lemma-p-q-quasicase}
In the foregoing notation we have that for every $a\in A$
\begin{equation}\label{p-q-def-quasicase}
\begin{split}
\sum p^1\otimes p^2(a) &=\sum \varphi _{1}^{1}\psi ^{1}\otimes \varphi _{2}^{1}\psi ^{2}\Phi
^{1}S\left( a\varphi ^{2}\psi _{1}^{3}\Phi ^{2}\right) \varphi ^{3}\psi
_{2}^{3}\Phi ^{3},  \\
\sum q^1(a)\otimes q^2 &=\sum \Phi ^{1}\varphi _{1}^{1}\psi ^{1}S\left( \Phi ^{2}\varphi
_{2}^{1}\psi ^{2} a\right) \Phi ^{3}\varphi ^{2}\psi _{1}^{3}\otimes \varphi
^{3}\psi _{2}^{3}.
\end{split}
\end{equation}
Moreover, the following relations hold for every $a,b\in A$
\begin{align}
\sum p^1a\otimes p^2(b) & = \sum a_{11}p^1\otimes a_{12}p^2(ba_2), \label{eq:pInvariant} \\
\sum q^1(a)\otimes bq^2 & = \sum q^1(b_1a)b_{21}\otimes q^2b_{22}. \label{eq:qInvariant}
\end{align}
\end{lemma*}

\begin{proof}
The reassociator $\Phi $ satisfies the dual relation to \eqref{eq:3-cocycle}, i.e.
\begin{equation*}
(1\otimes \Phi ) \cdot (A\otimes \Delta \otimes A)\left( \Phi \right) \cdot \left(\Phi \otimes
1\right)= (A\otimes A\otimes
\Delta )\left( \Phi \right) \cdot (\Delta \otimes A\otimes A)\left( \Phi \right).
\end{equation*}
In particular, it satisfies
\begin{equation*}
\sum \varphi _{1}^{1}\psi ^{1}\otimes \varphi _{2}^{1}\psi ^{2}\Phi
^{1}\otimes \varphi ^{2}\psi _{1}^{3}\Phi ^{2}\otimes \varphi ^{3}\psi
_{2}^{3}\Phi ^{3}=\sum \varphi ^{1}\psi ^{1}\otimes \varphi ^{2}\psi
_{1}^{2}\otimes \varphi ^{3}\psi _{2}^{2}\otimes \psi ^{3}.
\end{equation*}
Applying $\left( A\otimes m\right) \, \left( A\otimes A\otimes m\right)
\, \left( A\otimes A\otimes (S\leftharpoonup a)\otimes A\right) $ to both sides we get
\begin{align*}
& \sum \varphi _{1}^{1}\psi ^{1}\otimes \varphi _{2}^{1}\psi ^{2}\Phi
^{1}S\left(a \varphi ^{2}\psi _{1}^{3}\Phi ^{2}\right) \varphi ^{3}\psi
_{2}^{3}\Phi ^{3} \\
&  = \sum \varphi ^{1}\psi ^{1}\otimes \varphi ^{2}\psi
_{1}^{2}S\left(a \varphi ^{3}\psi _{2}^{2}\right) \psi ^{3} \stackrel{\eqref{eq:qbP}}{=} \sum \varphi ^{1}\otimes \varphi ^{2}S\left(a \varphi ^{3}\right) = \sum p^1\otimes p^2(a),
\end{align*}
which is the first identity in \eqref{p-q-def-quasicase}. The second one is proved analogously. Let us check that \eqref{eq:pInvariant} holds as well (\eqref{eq:qInvariant} is proved similarly). We compute
\begin{align*}
\sum p^1a\otimes p^2(b) & \stackrel{\eqref{eq:Defp}}{=} \sum \varphi^1a\otimes \varphi^2S(b\varphi^3) \stackrel{\eqref{eq:qbP}}{=} \sum \varphi^1a_1\otimes \varphi^2a_{21}S(b\varphi^3a_{22}) \\
& \stackrel{(*)}{=} \sum a_{11}\varphi^1\otimes a_{12}\varphi^2S(ba_{2}\varphi^3) = \sum a_{11}p^1\otimes a_{12}p^2(ba_2),
\end{align*}
where in $(*)$ we used the quasi-coassociativity $\Phi\cdot (\Delta\otimes A)\Delta=(A\otimes \Delta)\Delta\cdot \Phi$.
\end{proof}

\begin{lemma*}
\label{p-q-deformation-quasicase}Let $(A,m,u,\Delta ,\varepsilon ,\Phi ,S)$
be a quasi-bialgebra with preantipode and let $p,q$ be defined
as above. For all $a\in A$ we have that
\begin{equation*}
S(a)=\sum q^{1}(1)S\left( p^{1}aq^{2}\right) p^{2}(1)=\sum S\left( \varphi
^{1}\right) \varphi ^{2}S\left( \psi ^{1}a\varphi ^{3}\right) \psi
^{2}S\left( \psi ^{3}\right) .
\end{equation*}
\end{lemma*}

\begin{proof}
Keeping in mind that $\Phi^{-1}$ is counital, i.e. that it satisfies
\begin{equation*}
\left(\varepsilon\otimes A\otimes A\right)\left(\Phi^{-1}\right)=1\otimes 1=\left( A\otimes \varepsilon\otimes A\right)\left(\Phi^{-1}\right)=1\otimes 1=\left(A\otimes A\otimes \varepsilon\right)\left(\Phi^{-1}\right),
\end{equation*}
we may compute directly
\begin{align*}
& \sum  S\left( \varphi ^{1}\right) \varphi ^{2}S\left( \psi ^{1}a\varphi
^{3}\right) \psi ^{2}S\left( \psi ^{3}\right) =\sum q^{1}(1)S\left(
p^{1}aq^{2}\right) p^{2}(1) \\
\hspace{-1pt} & \hspace{-1pt} \stackrel{\eqref{p-q-def-quasicase}}{=} \hspace{-1pt} \sum \Phi ^{1}\varphi _{1}^{1}\psi
^{1}S\left( \Phi ^{2}\varphi _{2}^{1}\psi ^{2}\right) \Phi ^{3}\varphi ^{2}
\psi _{1}^{3}S\left( \gamma_{1}^{1}\phi
^{1}a\varphi ^{3}\psi _{2}^{3}\right) \gamma_{2}^{1}
\phi^{2}\Psi^{1}S\left( \gamma^{2}
\phi_{1}^{3}\Psi^{2}\right) \gamma
^{3}\phi_{2}^{3}\Psi^{3} \\
\hspace{-1pt} & \hspace{-1pt} \stackrel{\eqref{eq:qbP}}{=} \hspace{-1pt} \sum \Phi ^{1}\varphi _{1}^{1}S\left( \Phi ^{2}\varphi
_{2}^{1}\right) \Phi ^{3}\varphi ^{2}S\left( \phi^{1}a\varphi
^{3}\right) \phi^{2}\Psi^{1}S\left( 
\phi_{1}^{3}\Psi^{2}\right) \phi
_{2}^{3}\Psi^{3} \\
\hspace{-1pt} & \hspace{-1pt} \stackrel{\eqref{eq:qbP}}{=} \hspace{-1pt} \sum \Phi ^{1}S\left( \Phi ^{2}\right) \Phi ^{3}S\left( a\right) 
\Psi^{1}S\left( \Psi^{2}\right) \Psi
^{3}=S(a). \qedhere
\end{align*}
\end{proof}

\begin{proposition*}
Let $(A,m,u,\Delta ,\varepsilon ,\Phi ,S)$ be a quasi-bialgebra with
preantipode. For all $a,b\in A$ we have
\begin{equation}
S\left( ab\right) =\sum S\left( \varphi ^{1}b\right) \varphi ^{2}S\left(
\psi ^{1}\varphi ^{3}\right) \psi ^{2}S\left( a\psi ^{3}\right) .
\label{antimultpreanti}
\end{equation}
\end{proposition*}

\begin{proof}
We know from Lemma \ref{p-q-deformation-quasicase} that $S(a)=\sum q^1(1)S\left( p^{1}aq^{2}\right) p^2(1)$. Relation \eqref{antimultpreanti} is proved directly by applying it to $S\left( ab\right) $:
\begin{align*}
S\left( ab\right)  &\stackrel{\phantom{(52)}}{=}\sum q^1(1)S\left( p^{1}abq^{2}\right) p^2(1) \stackrel{\eqref{eq:pInvariant}}{=} \sum q^1(1)S\left( a_{11}p^{1}bq^{2}\right)a_{12} p^2(a_2) \\
& \stackrel{\eqref{eq:qbP}}{=}  \sum q^1(1)S\left(p^{1}bq^{2}\right) p^2(a) \stackrel{\eqref{eq:qInvariant}}{=} \sum q^1(b_1)b_{21}S\left(p^{1}q^{2}b_{22}\right) p^2(a)  \\
& \stackrel{\eqref{eq:qbP}}{=}  \sum q^1(b)S\left(p^{1}q^{2}\right) p^2(a) = \sum S\left( \varphi ^{1}b\right) \varphi ^{2}S\left( \psi ^{1}\varphi
^{3}\right) \psi ^{2}S\left( a\psi ^{3}\right). \qedhere
\end{align*}
\end{proof}

\noindent Formula \eqref{antimultpreanti} can be viewed as an
anti-multiplicativity of the preantipode.

\end{document}